\documentclass{amsart}

\usepackage{amsmath,amssymb,amsthm}
\usepackage{pinlabel}

\usepackage{rotating}

\newtheorem{prop}{Proposition}[section]
\newtheorem{thm}[prop]{Theorem}
\newtheorem{lem}[prop]{Lemma}

\theoremstyle{definition}

\newtheorem{defn}[prop]{Definition}

\newtheorem{rem}[prop]{Remark}

\def\co{\colon\thinspace}

\newcommand{\tb}{\mathtt{tb}}
\newcommand{\lk}{\mathtt{lk}}
\newcommand{\slk}{\mathtt{sl}}
\newcommand{\vlk}{\underline{\lk}}
\newcommand{\rot}{\mathtt{rot}}
\newcommand{\vrot}{\underline{\rot}}
\newcommand{\ttt}{\mathtt{t}}
\newcommand{\ttr}{\mathtt{r}}

\newcommand{\rmd}{\mathrm{d}}

\newcommand{\rmi}{\mathrm{i}}
\newcommand{\SO}{\mathrm{SO}}

\newcommand{\bfi}{\mathbf{i}}
\newcommand{\bfj}{\mathbf{j}}
\newcommand{\bfk}{\mathbf{k}}
\newcommand{\bfx}{\mathbf{x}}

\newcommand{\C}{\mathbb{C}}
\newcommand{\N}{\mathbb{N}}

\newcommand{\R}{\mathbb{R}}
\newcommand{\Z}{\mathbb{Z}}

\newcommand{\LL}{\mathbb{L}}

\newcommand{\SL}{\mathrm{SL}}

\newcommand{\xist}{\xi_{\mathrm{st}}}

\DeclareMathOperator{\Int}{Int}

\begin{document}

\author[H.~Geiges]{Hansj\"org Geiges}
\address{Mathematisches Institut, Universit\"at zu K\"oln,
Weyertal 86--90, 50931 K\"oln, Germany}
\email{geiges@math.uni-koeln.de}

\author[S.~Onaran]{Sinem Onaran}
\address{Department of Mathematics, Hacettepe University,
06800 Beytepe-Ankara, Turkey}
\email{sonaran@hacettepe.edu.tr}

\thanks{H.~G.\ is partially supported by the SFB/TRR 191 `Symplectic Structures
in Geometry, Algebra and Dynamics', funded by the DFG; S.~O.\
is partially supported by a Turkish Academy of Sciences T\"UBA--GEBIP}

\title{Legendrian Hopf links}

\date{}

\begin{abstract}
We completely classify Legendrian realisations
of the Hopf link, up to coarse equivalence, in the $3$-sphere with
any contact structure.
\end{abstract}



\maketitle


\section{Introduction}
When we speak of a \emph{Hopf link} in this paper, we shall
always mean an ordered link $K_0\sqcup K_1$ in the
$3$-sphere~$S^3$, made up
of oriented unknots forming a \emph{positive} Hopf link,
that is, $K_1$ is isotopic in $S^3\setminus K_0$
to a positive meridian of~$K_0$.
Two Legendrian realisations $L_0\sqcup L_1$ and $L_0'\sqcup L_1'$
of this Hopf link in some contact structure $\xi$ on $S^3$
are called \emph{coarsely equivalent} if there is a contactomorphism
of $(S^3,\xi)$ that sends $L_0\sqcup L_1$ to $L_0'\sqcup L_1'$
as an ordered, oriented link.

The main result of this paper is the classification, up to coarse
equivalence, of all Legendrian realisations of the Hopf link in $S^3$
with any contact structure. 

For a brief introduction to the theory of Legendrian knots, i.e.\
knots tangent to a given contact structure, see \cite[Chapter~3]{geig08}
or the beautiful survey by Etnyre~\cite{etny05}. The latter
discusses the classification of Legendrian knots
and covers a wide range of applications of Legendrian knot theory
not only to contact geometry (e.g.\ surgery along Legendrian knots,
invariants of contact structures), but also to general topology
(e.g.\ plane curves, knot concordance, topological knot invariants).

Very little is known about the classification of Legendrian links
(with more than one component).
When Etnyre wrote his survey on Legendrian and transverse
knots in 2005, the results about  Legendrian links could be
summarised on two pages. Since then, only a small number of classification
statements on Legendrian links in the standard tight contact
structure $\xist$ on $S^3$ or in other tight contact $3$-manifolds
have been added to the literature, e.g.\ \cite{dige07,dige10}.
The present paper goes considerably beyond those
results, both concerning the range of Legendrian realisations
covered by the classification and the variety of methods used in the proof.
Our main theorem is the first complete Legendrian classification
of a topological link type that includes Legendrian realisations
in overtwisted contact structures.

Legendrian knots in overtwisted contact structures fall into
two classes, loose and exceptional. The latter can be divided
into two subclasses.

\begin{defn}
A Legendrian knot $L$ in an overtwisted contact $3$-manifold $(M,\xi)$
is called \emph{exceptional} if its complement
$(M\setminus L,\xi|_{M\setminus L})$ is tight; the knot
is called \emph{loose} if the contact structure is still 
overtwisted when restricted to the knot complement.

An exceptional Legendrian knot is called \emph{strongly exceptional}
if the knot complement has zero Giroux torsion.
\end{defn}

The notion of \emph{Giroux torsion} and the related concept
of \emph{twisting} will be explained below. Previous classification
results for exceptional Legendrian knots such as
\cite{geon} confined attention to \emph{strongly} exceptional realisations.
One of the reasons for this restriction is that classification
results for tight contact structures on the relevant knot complement
typically tend to require, as in~\cite{dlz13},
that the Giroux torsion be zero. One of the significant
features of our classification of Legendrian Hopf
links, by contrast, is the fact that it
includes Legendrian realisations where the link complement
may contain torsion without being overtwisted.

Here is our main result. As mentioned before, $\xist$ denotes the
standard tight contact structure on~$S^3$. Besides this standard structure,
there is a countable family of overtwisted contact structures
$\bigl\{\xi_d\co d\in\Z+\frac{1}{2}\bigr\}$, as we shall recall below.
For the definition of the classical invariants $\tb$ (Thurston--Bennequin
invariant) and $\rot$ (rotation number) of Legendrian knots, as well as
other fundamentals of contact topology, the reader may refer
to~\cite{geig08}. Our natural numbers $\N$ are the positive integers;
$\N_0$ includes zero.

\begin{thm}
\label{thm:main}
Up to coarse equivalence, the Legendrian realisations of the
Hopf link are as follows. In all cases, the classical invariants
determine the Legendrian realisation.
\begin{itemize}
\item[(a)] In $(S^3,\xist)$ there is a unique realisation for any
combination of classical invariants $(\tb,\rot)=(\ttt_i,\ttr_i)$,
$i=0,1$, in the range $\ttt_0,\ttt_1<0$ and
\[ \ttr_i\in\{\ttt_i+1,\ttt_i+3,\ldots,-\ttt_i-3,-\ttt_i-1\}.\]
For fixed values of $\ttt_0,\ttt_1<0$ this gives a total
of $\ttt_0\ttt_1$ realisations.

\item[(b)] For $\ttt_0<0$ and $\ttt_1>0$ the strongly exceptional
realisations are as follows.
\begin{itemize}
\item[(b1)] In $(S^3,\xi_{1/2})$ there are realisations $L_0\sqcup L_1$
made up of an exceptional Legendrian unknot $L_1$ with invariants
$(\ttt_1,\ttr_1)=(n+2,\pm(n+1))$, where $n\in\N_0$, and a loose
Legendrian unknot~$L_0$ whose Thurston--Bennequin invariant $\ttt_0$
can be any negative number, and $\ttr_0$ lies in the range
\[ \{\ttt_0,\ttt_0+2,\ldots,\-\ttt_0-2,-\ttt_0\}.\]
For a given $\ttt_0<0$, this gives us $2|\ttt_0-1|$
realisations.

\item[(b2)] In $(S^3,\xi_{1/2})$ there are realisations $L_0\sqcup L_1$
consisting of the exceptional Legendrian unknot $L_1$ with invariants
$(\ttt_1,\ttr_1)=(1,0)$ and a loose Legendrian unknot~$L_0$,
where $\ttt_0$ can be any negative number, and $\ttr_0$ lies in the range
\[ \{\ttt_0-1,\ttt_0+1,\ldots,\-\ttt_0-1,-\ttt_0+1\}.\]
For a given $\ttt_0<0$, these are $|\ttt_0-2|$ realisations.
\end{itemize}

\item[(c)] For $\ttt_0,\ttt_1>0$ the strongly exceptional
realisations are as follows.

\begin{itemize}
\item[(c1)] There is a unique realisation in $(S^3,\xi_{1/2})$
with $(\ttt_i,\ttr_i)=(1,0)$, $i=0,1$. Both $L_0$ and $L_1$
are exceptional.
\item[(c2)] There is a pair of realisations with $(\ttt_0,\ttr_0)=(2,\pm 3)$
and $(\ttt_1,\ttr_1)=(1,\pm 2)$ in $(S^3,\xi_{-1/2})$.

There are three realisations with $\ttt_0=3$ and $\ttt_1=1$. Two
of them with $(\ttt_0,\ttr_0)=(3,\pm 4)$ and $(\ttt_1,\ttr_1)=(1,\pm 2)$
live in $(S^3,\xi_{-1/2})$; the third one with $(\ttt_0,\ttr_0)=(3,0)$
and $(\ttt_1,\ttr_1)=(1,0)$ can be found in $(S^3,\xi_{3/2})$.

There are four realisations with $\ttt_0=\ttt_1=2$: two
with $\ttr_1=\ttr_2=\pm 3$ in  $(S^3,\xi_{-1/2})$,
two with $\ttr_0=\ttr_1=\pm 1$ in $(S^3,\xi_{3/2})$.

In all cases, the individual link components are loose.
\item[(c3)] For any $\ttt_0\geq 4$ and $t_1=1$ there are four links
realising these values of the Thurston--Bennequin invariants.
For $\ttt_0\geq 3$ and $\ttt_1=2$ there are six realisations.
The remaining invariants are listed in Table~\ref{table:se-invariants}
in Section~\ref{section:compute}. The link components are loose.

\item[(c4)] For each choice of $\ttt_0,\ttt_1\geq 3$ there are 
eight realisations. In all cases both link components loose.
The invariants are listed in Table~\ref{table:se-invariants}.
\end{itemize}

\item[(d)] For $\ttt_0=0$, there are two exceptional realisations
with $\ttt_1=m$ for each $m\in\Z$. The rotation numbers are
$\ttr_0=\pm 1$ and $\ttr_1=\pm (m-1)$. The unknot $L_0$ is always loose;
$L_1$ is loose for $m\leq 0$, exceptional for $m\geq 1$.

\item[(e)] For each choice of integers $(\ttt_0,\ttt_1)\neq (\pm 1,\pm 1)$
and natural number $p$ there is exactly a pair of exceptional Legendrian Hopf
links $L_0\sqcup L_1$, distinguished by the rotation numbers,
with $\tb(L_i)=\ttt_i$ and with $\pi$-twisting
in the link complement equal to~$p$.
For $\ttt_0=\ttt_1=\pm 1$, there is a unique realisation. The ambient
contact structure is $\xi_{1/2}$ or~$\xi_{-1/2}$.

\item[(f)] For any choice of $\ttt_0,\ttr_0,\ttt_1,\ttr_1\in\Z$
with $\ttt_i+\ttr_i$ odd,
and for any $d\in\Z+\frac{1}{2}$, there is a unique loose Hopf link
$L_0\sqcup L_1$ in $(S^3,\xi_d)$ with invariants $\tb(L_i)=\ttt_i$ and
$\rot(L_i)=\ttr_i$.
\end{itemize}
\end{thm}

Explicit realisations will be exhibited below. Those
examples will give us a complete list of the classical invariants
that can be realised. Observe that exceptional realisations of
the Hopf link exist only in the three overtwisted structures
$\xi_{\pm 1/2}$ and~$\xi_{3/2}$.

For easier navigation, here is a guide to the paper, indicating where
each part of Theorem~\ref{thm:main} will be proved. Part (a) about
realisations in the tight contact structure, which was proved earlier
in~\cite{dige07}, will be discussed in Section~\ref{section:tight}.

The classification of strongly exceptional realisations, parts
(b) to~(d), is achieved in Section~\ref{section:exceptional},
and this takes up the largest part of the paper. In
Section~\ref{section:link-complement} we determine the number of tight
contact structures on the link complement as a function of
the values $\ttt_0,\ttt_1$ of the Thurston--Bennequin invariants,
using results of Giroux~\cite{giro00} and Honda~\cite{hond00I}.
This gives an upper bound on the number of Legendrian realisations.
We show that this bound is attained in all cases by exhibiting
explicit realisations in contact surgery diagrams. This strategy
was developed in~\cite{geon15} for the classification of
Legendrian rational unknots in lens spaces.

The classification of the Hopf links with twisting in the complement, part~(e),
will be given in Section~\ref{subsection:proofe}, based on the discussion
in the preceding parts of Section~\ref{section:twisting}. The necessary
preparations to describe explicit realisations in this case
are contained in Section~\ref{section:cut}, where we construct
a couple of overtwisted contact structures on $S^3$ as contact
cuts in the sense of Lerman~\cite{lerm01}. We recover some
results of Dymara~\cite{dyma04} about exceptional realisations
of the unknot in this model of~$S^3$, with considerably simplified
arguments.

Statement (f) about the classification of loose Legendrian Hopf links
will be proved in Section~\ref{section:loose}.
\section{Contact structures on $S^3$}
\label{section:contS3}
Throughout we are dealing with \emph{(co-)oriented}
and \emph{positive} contact structures on the $3$-sphere $S^3$, that is,
tangent $2$-plane fields $\xi$ that are described as
$\xi=\ker\alpha$ with some globally defined $1$-form $\alpha$ satisfying
$\alpha\wedge\rmd\alpha>0$ with respect to the standard
orientation of $S^3\subset\C^2$.

The standard contact structure
\begin{equation}
\label{eqn:xist}
\xist=\ker(x_1\,\rmd y_1-y_1\,\rmd x_1+x_2\,\rmd y_2-y_2\,\rmd x_2)
\end{equation}
on $S^3$ is the
unique tight contact structure, up to isotopy, on the $3$-sphere.
Furthermore, there is a countable family of overtwisted
contact structures. Their classification up to isotopy
coincides with their homotopy classification as tangent $2$-plane
fields.

There are two invariants that equally detect the homotopy class
of an oriented tangent $2$-plane field $\xi$ on~$S^3$. The first one is the
Hopf invariant~$h$. The definition of this invariant
presupposes that we fix a trivialisation
$TS^3\cong S^3\times\R^3$ of the tangent bundle of $S^3$.
The Gau{\ss} map of $\xi$ may then be regarded as
a map $S^3\rightarrow S^2$, which has a well-defined Hopf invariant~$h\in\Z$.

Alternatively, one may use the $d_3$-invariant introduced by
Gompf~\cite{gomp98}, cf.~\cite{dgs04,gost99}. This can be computed
from any compact almost complex $4$-manifold $(X,J)$
with boundary $\partial X=S^3$ such that the complex line
$TS^3\cap J(TS^3)$ in the tangent bundle $TS^3$ coincides
with the oriented plane field~$\xi$.
According to~\cite[Thm.~4.16]{gomp98},
the $d_3$-invariant is computed from this data as
\begin{equation}
\label{eqn:d3acs}
d_3(\xi)=\frac{1}{4}\bigl(c_1^2(X,J)-3\sigma(X)-2\chi(X)\bigr),
\end{equation}
where $c_1$ denotes the first Chern class, $\sigma$ the signature,
and $\chi$ the Euler characteristic of $(X,J)$.
Such an almost complex filling $(X,J)$ of $(S^3,\xi)$
can always be found, and $d_3(\xi)$ is independent of the choice
of filling. The $d_3$-invariant can be defined more generally for any
oriented tangent $2$-plane field on any closed, oriented $3$-manifold,
provided the Euler class of~$\xi$ is torsion. Notice that the
definition of $d_3(\xi)$ does not involve a choice of trivialisation
of the tangent bundle of the $3$-manifold in question.

For $S^3$ the $d_3$-invariant takes values in $\Z+\frac{1}{2}$,
see~\cite[Remark~2.6]{dgs04}.

\begin{rem}
Observe that the Hopf invariant $h(\xi)$ does not depend on the
choice of (co-)orientation of~$\xi$, since composition with the
antipodal map of $S^2$ does not change the Hopf invariant of a map
$S^3\rightarrow S^2$. The same is true for $d_3(\xi)$,
since $c_1(X,J)=-c_1(X,-J)$. This implies that on $S^3$
any oriented contact structure is (co-)orientation-reversingly
isotopic to itself. For $\xist$ as in (\ref{eqn:xist})
such an isotopy is given by a rotation through an angle $\pi$ in the
$x_1x_2$-plane; this isotopy carries over to suitable surgery descriptions
of the overtwisted contact structures.
\end{rem}

We shall also need the following formula \cite[Cor.~3.6]{dgs04}
for the $d_3$-invariant of a contact manifold $(Y,\xi)$
with $c_1(\xi)$ torsion that is obtained by contact
$(\pm 1)$-surgery in the sense of \cite{dige04}
along the oriented components of a Legendrian
link $\LL=\LL_+\sqcup\LL_-$, all of which have non-vanishing
Thurston--Bennequin invariant. In this situation,
\begin{equation}
\label{eqn:d3surgery}
d_3(\xi)=\frac{1}{4}\bigl( c^2-3\sigma (X)-2\chi (X)\bigr) +q;
\end{equation}
here $q$ denotes the number of components of $\LL_+$,
and $c\in H^2 (X)$ is the cohomology class
determined by $c(\Sigma_i)=\rot(L_i)$ for each $L_i\subset \LL$,
where $\Sigma_i\subset X$ is the oriented surface obtained by gluing
a Seifert surface of $L_i$ with the core disc
of the corresponding handle.

When we view $S^3$ as the unit sphere in the quaternions,
a natural trivialisation of $TS^3$ is provided by the basis
$\bfi p,\bfj p,\bfk p$ of $T_pS^3$, $p\in S^3$. With this choice
we have $h(\xist)=0$, since
$\xi_{\mathrm{st},p}$ is spanned by $\bfj p$ and $\bfk p$, so the Gau{\ss} map
$p\mapsto\bfi p$ of $\xist$ is the constant map with respect to this
trivialisation. This choice is understood in the following lemma.

\begin{lem}
\label{lem:h-d3}
The Hopf invariant $h$ and the $d_3$-invariant of oriented tangent $2$-plane
fields on $S^3$ are related by $d_3=-h-\frac{1}{2}$.
\end{lem}

\begin{proof}
The standard contact structure $\xist$ may be regarded as
the complex tangencies of $S^3\subset\C^2$, and the unit ball
in $\C^2$ constitutes an almost complex filling. Formula (\ref{eqn:d3acs})
then yields $d_3(\xist)=-\frac{1}{2}$, so the claimed relation
between $h$ and $d_3$ holds in this case.

In order to verify the relation in general, we consider
the effect of a $\pi$-Lutz twist along a transverse knot $K$
in $(S^3,\xist)$. As shown in \cite[p.~147]{geig08}
or \cite[p.~114]{elfr09}, the resulting
contact structure $\xi_K$ satisfies
\[ h(\xi_K)=\slk (K),\]
where $\slk(K)$ denotes the self-linking number of~$K$.

Let $L_{-1}$ be the standard Legendrian unknot in $(S^3,\xist)$
with $\tb(L_{-1})=-1$ and $\rot(L_{-1})=0$. Its positive transverse
push-off $K_{-1}$, by \cite[Prop.~3.5.36]{geig08}, has self-linking
number
\[ \slk(K_{-1})=\tb(L_{-1})-\rot(L_{-1})=-1.\]
Write $\xi_{-1}$ for the contact structure obtained by
a Lutz twist along $K_{-1}$, so that $h(\xi_{-1})=-1$.
According to \cite{dgs05}, performing
a Lutz twist along $K_{-1}$ has the same effect as
contact $(+1)$-surgeries along $L_{-1}$ and its Legendrian push-off with two
additional negative stabilisations.

Thus, the linking matrix of this surgery diagram is
\[ M=\begin{pmatrix}0&-1\\-1&2\end{pmatrix},\]
and the vector of rotation numbers equals
$\vrot=(0,-2)^{\ttt}$. The number $c^2$ is computed as
$\bfx^{\ttt}M\bfx$, where $\bfx$ is the solution of $M\bfx=\vrot$.
This yields $c^2=0$, and observing that $\sigma=0$ and $\chi=3$ we find
that the contact structure $\xi_{-1}$ satisfies
$d_3(\xi_{-1})=\frac{1}{2}$, which verifies the lemma for $\xi_{-1}$.

Similarly, we can find a Legendrian knot $L_1$ in $(S^3,\xist)$ with 
$\tb(L_1)=1$ and $\rot(L_1)=0$, e.g.\ a suitable Legendrian realisation
of the right-handed trefoil knot \cite[Figure~8]{etho01}. Its
positive transverse push-off $K_1$ has $\slk(K_1)=1-0=1$,
so a Lutz twist along $K_1$ yields a contact structure $\xi_1$
with $h(\xi_1)=1$. The corresponding surgery picture, by a
computation analogous to the one above, allows us to compute
$d_3(\xi_1)=-\frac{3}{2}$, which accords with our claim.

Under the disjoint (and unlinked) union of copies of $K_{-1}$ and $K_1$,
the self-linking number and hence the Hopf invariant of the
contact structure obtained by Lutz twists is additive.
The Lutz twists along such a disjoint union amounts
to a connected sum of the contact manifolds obtained by
individual Lutz twists. On the other hand, the
$d_3$-invariant of the connected sum of two contact structures
$\xi,\xi'$ on $S^3$ is given by
\[ d_3(\xi\#\xi')=d_3(\xi)+d_3(\xi')+\frac{1}{2},\]
see \cite[Lemma~4.2]{dgs04}. The formula $d_3=-h-\frac{1}{2}$
now follows in full generality.
\end{proof}

Since we are mostly working with surgery diagrams,
we shall in the sequel denote the overtwisted
contact structures on $S^3$ by their $d_3$-invariant, that is,
we shall write $\xi_{d}$ for the unique overtwisted contact structure
with $d_3(\xi_d)=d\in\Z+\frac{1}{2}$. There can be no confusion with
the notation $\xi_{\pm 1}$, using the Hopf invariant, in the present
section, since the values of the two invariants range over
disjoint sets.
\section{The link complement}
\label{section:link-complement}
The classification of tight contact structures on $T^2\times [0,1]$
is due to Giroux~\cite{giro00} and Honda~\cite{hond00I}.
In this section we use their results to find
the number of tight contact structures on the complement of a
Legendrian Hopf link $L_0\sqcup L_1$, in terms of the
Thurston--Bennequin invariant of the link components.

We think of $S^3$ as being decomposed into two solid tori
$V_0,V_1$, chosen as tubular neighbourhoods of $L_0,L_1$,
respectively, and a thickened torus $T^2\times[0,1]$, i.e.\
\[ S^3=V_0\cup_{\partial V_0=T^2\times\{0\}} T^2\times [0,1]
\cup_{T^2\times\{1\}=\partial V_1} V_1.\]
We write $\mu_i,\lambda_i$ for meridian and longitude
on $\partial V_i$, and we take the gluing in the decomposition above
to be given by
\begin{eqnarray*}
\mu_0     & = & S^1\times\{*\}\times\{0\},\\
\lambda_0 & = & \{*\}\times S^1\times\{0\},\\
\mu_1     & = & \{*\}\times S^1\times\{1\},\\
\lambda_1 & = & S^1\times\{*\}\times\{1\}.
\end{eqnarray*} 

Given a Legendrian Hopf link $L_0\sqcup L_1$ with
$\tb (L_i)=:\ttt_i$, $i=0,1$, we can choose $V_i$ as
a standard neighbourhood of~$L_i$, meaning that $\partial V_i$
is a convex surface with two dividing curves of slope
$1/\ttt_i$ with respect to the identification
of $\partial V_i$ with $\R^2/\Z^2$ defined by $(\mu_i,\lambda_i)$.

On $T^2\times[0,1]$ we measure slopes on the $T^2$-factor with
respect to $(\mu_0,\lambda_0)$. This means that in the
described situation we are dealing with a contact structure
on $T^2\times[0,1]$ with convex boundary, two dividing curves
on either boundary component, of slope $s_0=1/\ttt_0$ on
$T^2\times\{0\}$, and of slope $s_1=\ttt_1$ on $T^2\times\{1\}$.
Recall that a contact structure on $T^2\times[0,1]$
with these boundary conditions is called \emph{minimally twisting}
if every convex torus parallel to the boundary has slope between $s_1$
and~$s_0$.

The following proposition covers all possible pairs $(\ttt_0,\ttt_1)$,
possibly after exchanging the roles of $L_0$ and $L_1$.

\begin{prop}
\label{prop:complement}
Up to an isotopy fixing the boundary, the number $N=N(\ttt_0,\ttt_1)$
of tight, minimally twisting contact structures
on $T^2\times[0,1]$ with convex boundary, two dividing curves
on either boundary component of slope $s_0=1/\ttt_0$ and
$s_1=\ttt_1$, respectively, is as follows.

\begin{itemize}
\item[(a1)] If $\ttt_0,\ttt_1<0$, excluding the
case $\ttt_0=\ttt_1=-1$, we have $N=\ttt_0\ttt_1$.
\item[(a2)] If $\ttt_0=\ttt_1=-1$, there is a unique
structure up to diffeomorphism, and an integral family (distinguished
by a holonomy map) up to isotopy.
\item[(b1)] If $\ttt_0<0$ and $\ttt_1\geq 2$, then $N=2|\ttt_0-1|$.
\item[(b2)] If $\ttt_0<0$ and $\ttt_1=1$, then $N=|\ttt_0-2|$.
\item[(c1)] If $\ttt_0=\ttt_1=1$, there is a unique
structure up to diffeomorphism, and an integral family (distinguished
by a holonomy map) up to isotopy.
\item[(c2)] $N(2,1)=2$, $N(3,1)=3$, and $N(2,2)=4$.
\item[(c3)] For all $\ttt_0\geq 4$ we have $N(\ttt_0,1)=4$;
for all $\ttt_0\geq 3$ we have $N(\ttt_0,2)=6$.
\item[(c4)] For all $\ttt_0\geq\ttt_1\geq3$, we have $N(\ttt_0,\ttt_1)=8$.
\item[(d)] For all $\ttt_1\in\Z$, we have $N(0,\ttt_1)=2$.
\end{itemize}
\end{prop}

\begin{proof}
In all cases, we need to normalise the slopes by applying
an element of $\mathrm{Diff}^+(T^2)\simeq\SL(2,\Z)$ to $T^2\times[0,1]$
such that the slope on $T^2\times\{0\}$ becomes $s_0'=-1$,
and on $T^2\times\{1\}$ we have $s_1'\leq -1$. If $s_1'<-1$,
the number $N$ is found from a continuous fraction expansion
\[ s_1'=r_0-\cfrac{1}{r_1-\cfrac{1}{r_2-\cdots-\cfrac{1}{r_k}}}
=:[r_0,\ldots,r_k]\]
with all $r_i<-1$ as
\begin{equation}
\label{eqn:N}
N=|(r_0+1)\cdots(r_{k-1}+1)r_k|,
\end{equation}
see~\cite[Theorem~2.2(2)]{hond00I}.
The vector $\begin{pmatrix}x\\y\end{pmatrix}$ stands for
the curve $x\mu_0+y\lambda_0$, with slope $y/x$.

(a1) We have
\[ \begin{pmatrix}0&1\\-1&\ttt_0-1\end{pmatrix}
\begin{pmatrix}\ttt_0\\1\end{pmatrix}=
\begin{pmatrix}1\\-1\end{pmatrix}\;\;\;\text{and}\;\;\;
\begin{pmatrix}0&1\\-1&\ttt_0-1\end{pmatrix}
\begin{pmatrix}1\\\ttt_1\end{pmatrix}=
\begin{pmatrix}\ttt_1\\ \ttt_0\ttt_1-\ttt_1-1\end{pmatrix}.\]
This means
\begin{equation}
\label{eqn:slope1}
s_1'=\ttt_0-1-\frac{1}{\ttt_1}.
\end{equation}
If $\ttt_1\leq -2$, this is a continued fraction expansion
$[\ttt_0-1,\ttt_1]$ as required by~\cite{hond00I}, and by (\ref{eqn:N})
we have $N=\ttt_0\ttt_1$. If $\ttt_1=-1$, but $\ttt_2\leq-2$,
we have the continued fraction expansion $s_1'=[\ttt_0]$,
and again this gives $|\ttt_0|=\ttt_0\ttt_1$ structures.

(a2) If $\ttt_0=\ttt_1=-1$, we can apply the same
transformation as in~(a), and we are then in the situation $s_0'=s_1'=-1$
of \cite[Theorem~2.2(4)]{hond00I}, cf.~\cite[Theorem~6.1]{etny-class},
which gives the claimed number.

(b1) Using the transformation as in (a), we find the same $s_1'<-1$ as
in~(\ref{eqn:slope1}), but this is not, as it stands,
a continued fraction expansion of the required form. From the continued
fraction expansion
\begin{equation}
\label{eqn:fracp}
-\frac{p+1}{p}=[\underbrace{-2,\ldots,-2}_p] \;\;\text{for}\;\; p\in\N
\end{equation}
we find
\[ \ttt_0-1-\frac{1}{\ttt_1}=\ttt_0-\frac{\ttt_1+1}{\ttt_1}=
[\ttt_0-2,\underbrace{-2,\ldots,-2}_{\ttt_1-1}].\]
By (\ref{eqn:N}) this yields $N=2|\ttt_0-1|$.

(b2) In this case the transformed slope is $s_1'=\ttt_0-2$, so
the continued fraction expansion is $[\ttt_0-2]$, giving us
$|\ttt_0-2|$ structures by~(\ref{eqn:N}).

(c) We consider the transformation
\[ \begin{pmatrix}1&1-\ttt_0\\-2&-1+2\ttt_0\end{pmatrix}
\begin{pmatrix}\ttt_0\\1\end{pmatrix}=
\begin{pmatrix}1\\-1\end{pmatrix}\]
and
\[\begin{pmatrix}1&1-\ttt_0\\-2&-1+2\ttt_0\end{pmatrix}
\begin{pmatrix}1\\\ttt_1\end{pmatrix}=
\begin{pmatrix}1+\ttt_1-\ttt_0\ttt_1\\-2-\ttt_1+2\ttt_0\ttt_1\end{pmatrix}.\]
This gives
\[ s_1'=\frac{-2-\ttt_1+2\ttt_0\ttt_1}{1+\ttt_1-\ttt_0\ttt_1}=
-2+\frac{\ttt_1}{1+\ttt_1-\ttt_0\ttt_1},\]
which is smaller than $-1$ for $\ttt_0\geq\ttt_1>0$, except
in the cases $(\ttt_0,\ttt_1)=(1,1)$ and $(\ttt_0,\ttt_1)=(2,1)$,
when $s_1'=-1$ or $s_1'=\infty$, respectively.

(c1) For $(\ttt_0,\ttt_1)=(1,1)$ the argument is now as in the case~(a2).

(c2) For $(\ttt_0,\ttt_1)=(2,2)$ we have $s_1'=-4$, which by
(\ref{eqn:N}) gives $N=4$. For $\ttt_1=1$ and $\ttt_0=3$ we have $s_1'=-3$,
and hence $N=3$. For $(\ttt_0,\ttt_1)=(2,1)$ we need to choose
a different transformation to obtain $s_1'\leq-1$. Such a transformation is
given by
\[ \begin{pmatrix}2&-3\\-3&5\end{pmatrix}
\begin{pmatrix}2\\1\end{pmatrix}=
\begin{pmatrix}1\\-1\end{pmatrix}\;\;\;\text{and}\;\;\;
\begin{pmatrix}2&-3\\-3&5\end{pmatrix}
\begin{pmatrix}1\\1\end{pmatrix}=
\begin{pmatrix}-1\\2\end{pmatrix},\]
which yields $s_1'=-2$ and hence $N=2$.

(c3) For $\ttt_0\geq4$, we have the continued fraction expansion
\[ s_1'=\frac{-3+2\ttt_0}{2-\ttt_0}=-3+\frac{\ttt_0-3}{\ttt_0-2}=
-3-\frac{1}{-\frac{\ttt_0-2}{\ttt_0-3}}=
[-3,\underbrace{-2,\ldots,-2}_{\ttt_0-3}]\]
by~(\ref{eqn:fracp}), which by (\ref{eqn:N}) yields $N=2\cdot 2=4$.

For $\ttt_0\geq 3$ and $\ttt_1=2$ we find
\[ s_1'=\frac{4\ttt_0-4}{3-2\ttt_0}=
[-3,\underbrace{-2,\ldots,-2}_{\ttt_0-3},-3].\]
In order to verify this continued fraction expansion, one may observe that
\begin{equation}
\label{eqn:23}
[\underbrace{-2,\ldots,-2}_p,-3]=-\frac{2p+3}{2p+1},
\end{equation}
which is easily proved by induction. Given this expansion for~$s_1'$, with
(\ref{eqn:N}) we obtain $N=2\cdot 3=6$.

(c4) The formula (\ref{eqn:23}) can be generalised to
\[ [\underbrace{-2,\ldots,-2}_p,\frac{a}{b}]:=
-2-\cfrac{1}{-2-\cfrac{1}{-2-\cdots-\cfrac{1}{a/b}}}=
-\frac{(p+1)a+pb}{pa+(p-1)b}.\]
It is then straightforward to verify that
\[ s_1'=\frac{-2-\ttt_1+2\ttt_0\ttt_1}{1+\ttt_1-\ttt_0\ttt_1}=
[-3,\underbrace{-2,\ldots,-2}_{\ttt_0-3},-3,
\underbrace{-2,\ldots,-2}_{\ttt_1-2}].\]
With (\ref{eqn:N}) this yields
\[ N=|(-3+1)\cdot(-3+1)\cdot (-2)|=8.\]

(d1) For $\ttt_0=0$ and $\ttt_1>0$ we use the transformation
\[\begin{pmatrix}0&1\\-1&-1\end{pmatrix}
\begin{pmatrix}0\\1\end{pmatrix}=
\begin{pmatrix}1\\-1\end{pmatrix}\;\;\;\text{and}\;\;\;
\begin{pmatrix}0&1\\-1&-1\end{pmatrix}
\begin{pmatrix}1\\ \ttt_1\end{pmatrix}=
\begin{pmatrix}\ttt_1\\-1-\ttt_1\end{pmatrix},\]
giving us
\[ s_1'=-\frac{\ttt_1+1}{\ttt_1}=[\underbrace{-2,\ldots,-2}_{\ttt_1}]\]
by~(\ref{eqn:fracp}), and hence $N=2$ by~(\ref{eqn:N}).

(d2) For $\ttt_0=0$ and $\ttt_1<0$, the transformation
\[\begin{pmatrix}-\ttt_1+1&1\\ \ttt_1-2&-1\end{pmatrix}
\begin{pmatrix}0\\1\end{pmatrix}=
\begin{pmatrix}1\\-1\end{pmatrix}\;\;\;\text{and}\;\;\;
\begin{pmatrix}-\ttt_1+1&1\\ \ttt_1-2&-1\end{pmatrix}
\begin{pmatrix}1\\ \ttt_1\end{pmatrix}=
\begin{pmatrix}1\\-2\end{pmatrix}\]
gives us $s_1'=-2$, and hence $N=2$ by~(\ref{eqn:N}).

(d3) For $\ttt_0=\ttt_1=0$ we use the transformation
\[\begin{pmatrix}1&1\\-2&-1\end{pmatrix}
\begin{pmatrix}0\\1\end{pmatrix}=
\begin{pmatrix}1\\-1\end{pmatrix}\;\;\;\text{and}\;\;\;
\begin{pmatrix}1&1\\-2&-1\end{pmatrix}
\begin{pmatrix}1\\0\end{pmatrix}=
\begin{pmatrix}1\\-2\end{pmatrix}.\]
Once again, this yields $s_1'=-2$ and $N=2$.
\end{proof}

When a tight contact structure on $T^2\times[0,1]$ is not minimally
twisting, one can associate with it a natural number, called the
\emph{$\pi$-twisting in the $[0,1]$-direction}~\cite[Section~2.2.1]{hond00I},
or simply \emph{twisting}.

There is also a notion of torsion for contact structures,
introduced by Giroux~\cite{giro94,giro99}. Let $(M,\xi)$ be a contact
$3$-manifold and $[T]$ an isotopy class of embedded $2$-tori in~$M$.
Then the \emph{$\pi$-torsion} (or simply \emph{torsion})
of $(M,\xi)$ is the supremum of $n\in\N_0$ for which there is
a contact embedding of
\[ \bigl( T^2\times[0,1],\ker(\sin(n\pi z)\,\rmd x+\cos(n\pi z)\,\rmd y)
\bigr) \]
into $(M,\xi)$, with $T^2\times\{z\}$ being sent to the class~$[T]$.

As explained in \cite[p.~86]{hond00II},
the twisting of a contact structure on $T^2\times[0,1]$
equals its torsion with respect to the class $[T^2\times\{z\}]$.
For the computation of the torsion, it is assumed that
the characteristic foliation on the boundary tori is a
non-singular foliation of some rational slope; the twisting is computed
for a convex boundary with dividing curves of the same slope, obtained
by a slight perturbation of the boundary tori.

The following isotopy classification can be found in
\cite[Theorem~2.2]{hond00I}. The diffeomorphism classification
for $s_1=-1$ can be deduced from the explicit description
of these structures in \cite[Lemma~5.2]{hond00I}. The fact that
in all other cases there are two structures even up to diffeomorphism
is a consequence of having two Legendrian realisations of Hopf links
whose complement has such boundary data, see Section~\ref{section:twisting}.

\begin{prop}
\label{prop:complement-twist}
Up to an isotopy fixing the boundary, the number of tight
contact structures on $T^2\times[0,1]$ with convex boundary,
two dividing curves on either boundary of slope $s_0=-1$ and $s_1\leq -1$
and positive twisting $n\in\N$, equals two for each~$n$. Up to
a diffeomorphism fixing the boundary, the number is likewise two,
except in the case $s_1=-1$, when it equals one.
\qed
\end{prop}

\begin{rem}
The two contact structures with a given positive twisting
and the same boundary data differ only by the choice of coorientation.
For $s_1=-1$, a diffeomorphism changing the coorientation is given
by $T^2\rightarrow T^2$, $(x,y)\mapsto (-y,x)$.
\end{rem}

By transforming the slopes as in the proof of
Proposition~\ref{prop:complement}, we see that
Proposition~\ref{prop:complement-twist} applies likewise to contact
structures with positive twisting and any
combination of boundary slopes $s_0=1/\ttt_0$ and $s_1=\ttt_1$.
\section{Hopf links in $(S^3,\xist)$}
\label{section:tight}
The classification of Legendrian Hopf links in the tight contact
structure $\xist$ on~$S^3$, up to Legendrian
isotopy, was carried out in \cite{dige07}
as part of a more general study of Legendrian cable links.
As shown there, these links are classified by their
classical invariants, and the range of these invariants
is the same for each component as for a single Legendrian unknot,
i.e.\ the Thurston--Bennequin invariants of the two
components can be any pair of negative integers, and for
$\tb(L_i)=-m$ the rotation number of $L_i$ can take any value
in the set
\[ \{-m+1,-m+3,\ldots, m-3, m-1\}.\]
Explicit realisations are given by stabilising the components
of the Hopf link shown (in the front projection) in
Figure~\ref{figure:hopf-link}, where the two components have
$\tb=-1$ and $\rot=0$.

\begin{figure}[h]
\centering
\includegraphics[scale=1]{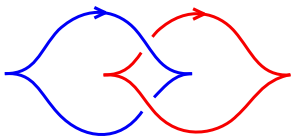}
  \caption{Legendrian Hopf link in $(S^3,\xist)$.}
  \label{figure:hopf-link}
\end{figure}

Here is an alternative proof of
this result. For given values of $\ttt_0,\ttt_1<0$, we have
$\ttt_0\ttt_1$ explicit realisations in $(S^3,\xist)$, which is
the maximal number possible by Proposition~\ref{prop:complement}~(a).
There are no realisations in $(S^3,\xist)$ with one of the
$\ttt_i$ being non-negative, since Legendrian unknots
in $(S^3,\xist)$ satisfy $\tb<0$ by the
Bennequin inequality \cite[Theorem~4.6.36]{geig08}. Also,
there are no realisations with twisting in the complement, since
this would force the corresponding contact structure on $S^3$ to be
overtwisted.

This proves part (a) of Theorem~\ref{thm:main}.
\section{Strongly exceptional Hopf links}
\label{section:exceptional}
In this section we classify the Legendrian realisations of the
Hopf link in overtwisted contact structures whose link
complement is tight and minimally twisting.
\subsection{Kirby moves}
We begin with some examples of Kirby diagrams of the Hopf link
that will be relevant in several cases of this classification.

\begin{lem}
\label{lem:kirby}
(i) The oriented link $L_0\sqcup L_1$ in the surgery diagram
shown in the first line of Figure~\ref{figure:b1kirby}
is a positive or negative Hopf link in~$S^3$, depending
on $n$ being even or odd.

(ii) The same is true for the link shown
in the first line of Figure~\ref{figure:c3-t0-1kirby}.

(iii) The oriented link $L_0\sqcup L_1$ on the
top left of Figure~\ref{figure:b2kirby}
is a positive Hopf link; the same holds for the links
in Figures \ref{figure:c2-3-1kirby} and~\ref{figure:dkirby}.
\end{lem}

\begin{proof}
For (i) and (iii) this follows from the Kirby moves shown in the corresponding
figure. For (ii) we observe that the Kirby moves in
Figure~\ref{figure:c3-t0-1kirby} reduce this to the situation in~(i),
with $n$ replaced by $n+2$. (The roles of $L_0$
and $L_1$ are exchanged in this diagram compared with
Figure~\ref{figure:b1kirby}; this choice conforms with the
Legendrian realisations discussed below.)
\end{proof}

\begin{figure}[h]
\labellist
\small\hair 2pt
\pinlabel $-1$ [b] at 34 221
\pinlabel $-2$ at 24 230
\pinlabel $-2$ [bl] at 49 240
\pinlabel $-2$ [bl] at 60 240
\pinlabel $-2$ [bl] at 95 240
\pinlabel $L_1$ [br] at 22 240
\pinlabel $L_0$ [bl] at 116 240
\pinlabel $n$ [t] at 75 208
\pinlabel $1$ [r] at 18 184
\pinlabel $1$ [r] at 18 176
\pinlabel $L_1$ [br] at 29 190
\pinlabel $1$ [b] at 40 170
\pinlabel $-2$ [bl] at 55 189
\pinlabel $-2$ [bl] at 66 189
\pinlabel $-2$ [bl] at 99 189
\pinlabel $L_0$ [bl] at 120 189
\pinlabel $L_1$ [r] at 23 142
\pinlabel $1$ [r] at 22 136
\pinlabel $1$ [r] at 22 129
\pinlabel $1$ [b] at 39 146
\pinlabel $-2$ [bl] at 55 144
\pinlabel $-2$ [bl] at 67 144
\pinlabel $-2$ [bl] at 100 144
\pinlabel $L_0$ [bl] at 121 144
\pinlabel $L_1$ [br] at 21 100
\pinlabel $-1$ [b] at 44 103
\pinlabel $-2$ [bl] at 58 100
\pinlabel $-2$ [bl] at 72 100
\pinlabel $-2$ [bl] at 106 100
\pinlabel $L_0$ [bl] at 126 100
\pinlabel $L_1$ [br] at 34 58
\pinlabel $-1$ [b] at 57 60
\pinlabel $-2$ [b] at 72 60
\pinlabel $-2$ [b] at 106 60
\pinlabel $L_0$ [bl] at 124 60
\pinlabel $L_1$ [br] at 38 20
\pinlabel $L_0$ [bl] at 71 20
\pinlabel $L_1$ [br] at 99 20
\pinlabel $L_0$ [bl] at 132 20
\pinlabel ${\text{$n$ odd}}$ [r] at 32 12
\pinlabel ${\text{$n$ even}}$ [l] at 138 12
\endlabellist
\centering
\includegraphics[scale=1.5]{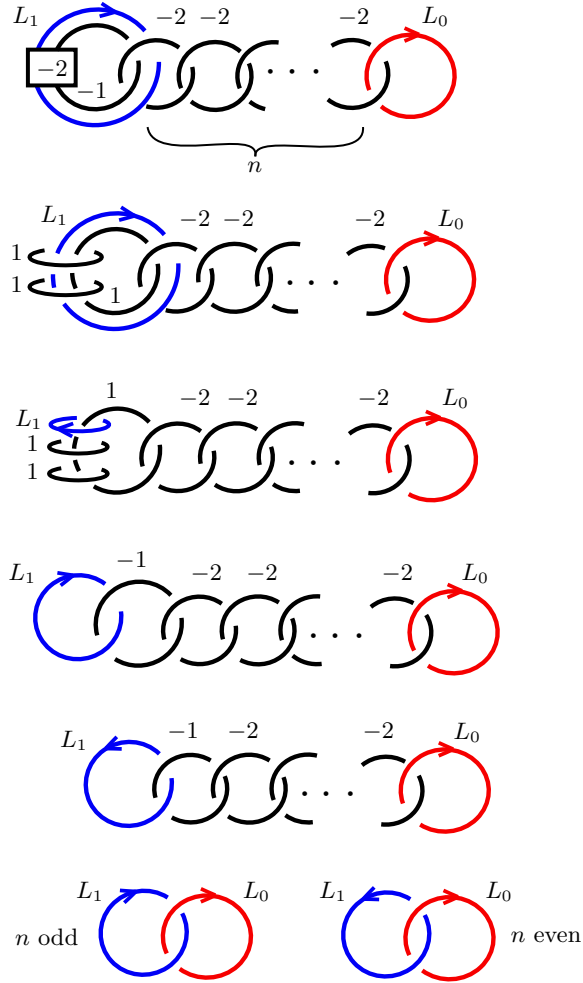}
  \caption{Kirby diagram of a Hopf link I.}
  \label{figure:b1kirby}
\end{figure}

\begin{figure}[h]
\labellist
\small\hair 2pt
\pinlabel $L_0$ [br] at 3 71
\pinlabel $L_1$ [bl] at 11 52
\pinlabel $-1$ at 33 59
\pinlabel $0$ [bl] at 62 69
\pinlabel $0$ [tr] at 56 66
\pinlabel $L_0$ [br] at 84 71
\pinlabel $L_1$ [bl] at 92 52
\pinlabel $1$ [r] at 100 62
\pinlabel $1$ [bl] at 144 69
\pinlabel $1$ [tr] at 137 66
\pinlabel $L_0$ [br] at 18 22
\pinlabel $L_1$ [bl] at 64 22
\pinlabel $-1$ [tr] at 29 13
\pinlabel $L_0$ [br] at 88 22
\pinlabel $L_1$ [bl] at 123 18
\endlabellist
\centering
\includegraphics[scale=1.8]{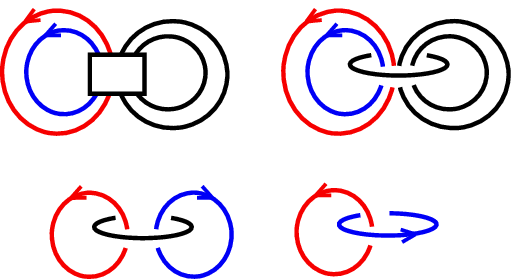}
  \caption{Kirby diagram of a Hopf link II.}
  \label{figure:b2kirby}
\end{figure}

\begin{figure}[h]
\labellist
\small\hair 2pt
\pinlabel $L_0$ [br] at 22 76
\pinlabel $-2$ at 21 60
\pinlabel $-2$ [t] at 34 72
\pinlabel $-1$ at 53 60
\pinlabel $L_1$ [bl] at 76 76
\pinlabel $0/0/0$ [t] at 66 43
\pinlabel $L_0$ [br] at 104 76
\pinlabel $1$ [br] at 94 67
\pinlabel $1$ [tr] at 94 60
\pinlabel $1$ [br] at 123 64
\pinlabel $1$ [b] at 116 51
\pinlabel $L_1$ [bl] at 160 76
\pinlabel $1/1/1$ [t] at 149 43
\pinlabel $L_0$ [b] at 3 23
\pinlabel $1$ [r] at 0 16
\pinlabel $1$ [r] at 0 9
\pinlabel $1$ [bl] at 28 22
\pinlabel $-2$ [tl] at 43 13
\pinlabel $L_1$ [bl] at 56 22
\pinlabel $L_0$ [br] at 78 22
\pinlabel $L_1$ [bl] at 124 22
\pinlabel $-1$ [tl] at 112 14
\pinlabel $L_0$ [br] at 149 22
\pinlabel $L_1$ [b] at 179 21
\endlabellist
\centering
\includegraphics[scale=1.8]{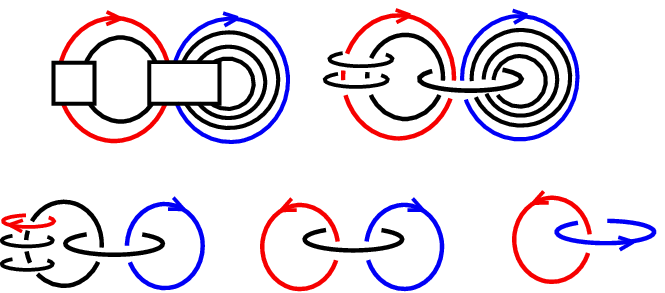}
  \caption{Kirby diagram of a Hopf link III.}
  \label{figure:c2-3-1kirby}
\end{figure}

\begin{figure}[h]
\labellist
\small\hair 2pt
\pinlabel $L_0$ [br] at 12 148
\pinlabel $-2$ at 11 133
\pinlabel $-1$ [b] at 24 123
\pinlabel $-2$ [b] at 40 143
\pinlabel $-2$ [b] at 52 143
\pinlabel $-2$ [b] at 67 143
\pinlabel $-2$ [b] at 86 143
\pinlabel $-3$ [b] at 101 143
\pinlabel $-1$ at 117 133
\pinlabel $L_1$ [bl] at 141 148
\pinlabel $0/0/0$ [t] at 129 116
\pinlabel $n$ [t] at 65 114
\pinlabel $L_0$ [br] at 12 95
\pinlabel $1$ [r] at 0 85
\pinlabel $1$ [r] at 0 79
\pinlabel $1$ [b] at 24 69
\pinlabel $-2$ [b] at 39 89
\pinlabel $-2$ [b] at 52 89
\pinlabel $-2$ [b] at 67 89
\pinlabel $-2$ [b] at 86 89
\pinlabel $-2$ [b] at 101 89
\pinlabel $1$ [tr] at 104 77
\pinlabel $L_1$ [bl] at 141 95
\pinlabel $1/1/1$ [t] at 129 62
\pinlabel $n$ [t] at 64 60
\pinlabel $L_0$ [r] at 7 31
\pinlabel $1$ [r] at 7 25
\pinlabel $1$ [r] at 7 18
\pinlabel $1$ [b] at 25 37
\pinlabel $-2$ [b] at 37 33
\pinlabel $-2$ [b] at 52 33
\pinlabel $-2$ [b] at 67 33
\pinlabel $-2$ [b] at 86 33
\pinlabel $-2$ [b] at 101 33
\pinlabel $-2$ [tl] at 124 22
\pinlabel $L_1$ [bl] at 136 31
\pinlabel $n+2$ [t] at 75 0
\endlabellist
\centering
\includegraphics[scale=1.6]{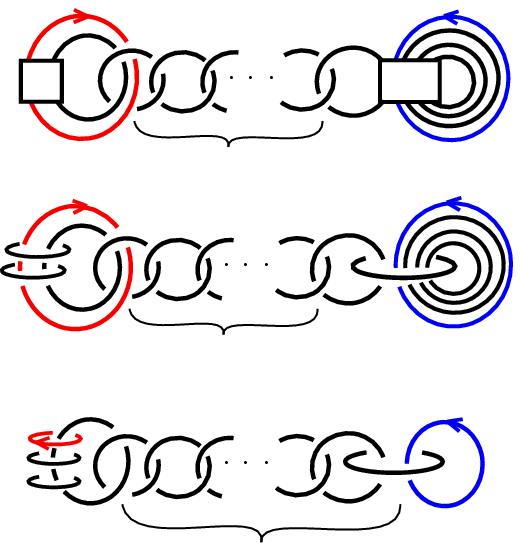}
  \caption{Kirby diagram of a Hopf link IV.}
  \label{figure:c3-t0-1kirby}
\end{figure}

\begin{figure}[h]
\labellist
\small\hair 2pt
\pinlabel $L_1$ [br] at 49 149
\pinlabel $-1$ at 51 135
\pinlabel $-1$ [b] at 64 126
\pinlabel $-1$ at 83 135
\pinlabel $L_0$ [r] at 103 136
\pinlabel $n$ [t] at 112 132
\pinlabel $0/\ldots/0$ [t] at 99 114
\pinlabel $L_1$ [br] at 135 149
\pinlabel $1$ [br] at 130 139
\pinlabel $1$ [b] at 150 126
\pinlabel $1$ [t] at 166 133
\pinlabel $L_0$ [b] at 185 128
\pinlabel $n$ [t] at 201 132
\pinlabel $1/\ldots/1$ [t] at 188 114
\pinlabel $L_1$ [r] at 2 89
\pinlabel $1$ [r] at 0 80
\pinlabel $1$ [b] at 20 95
\pinlabel $1$ [t] at 33 78
\pinlabel $L_0$ [b] at 52 72
\pinlabel $n$ [t] at 69 76
\pinlabel $1/\ldots/1$ [t] at 55 60
\pinlabel $L_1$ [r] at 91 89
\pinlabel $0$ [b] at 108 95
\pinlabel $-1$ [r] at 111 81
\pinlabel $L_0$ [b] at 141 72
\pinlabel $n-2$ [tl] at 156 77
\pinlabel $1/\ldots/1$ [t] at 143 60
\pinlabel $L_1$ [r] at 180 89
\pinlabel $1$ [b] at 198 94
\pinlabel $1$ at 214 80
\pinlabel $L_0$ [r] at 235 81
\pinlabel $n-2$ [tl] at 244 77
\pinlabel $2/\ldots/2$ [t] at 230 60
\pinlabel $L_1$ [br] at 6 24
\pinlabel $1$ [t] at 23 13
\pinlabel $1$ at 31 26
\pinlabel $L_0$ [tr] at 50 27
\pinlabel $n-2$ [tl] at 58 20
\pinlabel $2/\ldots/2$ [t] at 46 2
\pinlabel $L_1$ [br] at 84 24
\pinlabel $-1$ at 99 17
\pinlabel $1$ at 103 28
\pinlabel $L_0$ [r] at 123 25
\pinlabel $n-2$ [tl] at 131 22
\pinlabel $1/\ldots/1$ [t] at 117 2
\pinlabel $L_1$ [br] at 162 36
\pinlabel $L_0$ [tl] at 190 27
\pinlabel $n-2\;\text{times}\;1$ [t] at 168 4
\pinlabel $L_1$ [br] at 217 36
\pinlabel $L_0$ [tl] at 245 24
\endlabellist
\centering
\includegraphics[scale=1.3]{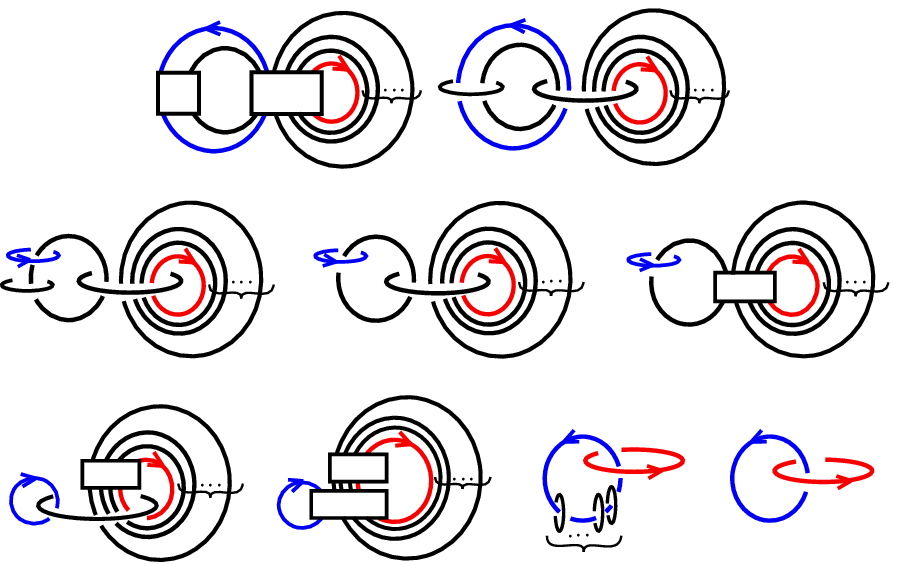}
  \caption{Kirby diagram of a Hopf link V.}
  \label{figure:dkirby}
\end{figure}
\subsection{Computing the invariants}
We are going to describe Legendrian realisations $L_0\sqcup L_1$ of the
Hopf link in $(S^3,\xi_d)$ as front projections in a contact surgery diagram
for this contact manifold. Here we briefly summarise how to
compute the classical invariants in this setting.

We number the Legendrian knots in the contact surgery diagram as
$K_1,\ldots, K_n$ and choose
auxiliary orientations on them. Write $M$
for the corresponding linking matrix, with the diagonal entries
given by the topological surgery framing. Given the presentation
of a Legendrian knot $L_i$ in the surgery diagram, we form the
extended linking matrix
\[ M_i=\left(\begin{array}{c|ccc}
0            & \lk(L_i,K_1) & \cdots & \lk(L_i,K_n)\\ \hline
\lk(L_i,K_1) &              &        &             \\
\vdots       &              & M      &             \\
\lk(L_i,K_n) &              &        &
\end{array}\right). \]
\subsubsection{Thurston--Bennequin invariant}
Write $\tb_i$ for the Thurston--Bennequin invariant of $L_i$
as a knot in $(S^3,\xist)$, before performing the contact surgeries
along $K_1,\ldots,K_n$. Then, as shown in \cite[Lemma~6.6]{loss09},
the Thurston-Bennequin invariant $\tb(L_i)$ in the contact
structure on $S^3$ obtained by contact surgeries along $K_1,\ldots,K_n$ 
is
\begin{equation}
\label{eqn:tb}
\tb(L_i)=\tb_i+\frac{\det M_i}{\det M}.
\end{equation}
Alternatively, one can keep track of the contact framing of $L_i$
during the topological Kirby moves that turn the given surgery diagram
into the empty diagram, and $L_0\sqcup L_1$ into the standard Hopf
link in this empty diagram for~$S^3$.
\subsubsection{Rotation number}
By $\rot_i$ we denote the rotation number of $L_i$ before the
surgery. Write
\[ \vrot=\bigl(\rot(K_1),\ldots,\rot(K_n)\bigr)\]
for the vector of rotation numbers of the surgery knots,
and
\[ \vlk_i=\bigl(\lk(L_i,K_1),\ldots,\lk(L_i,K_n)\bigr)\]
for the vector of linking numbers. Then, again by~\cite[Lemma~6.6]{loss09},
the rotation number $\rot(L_i)$ of $L_i$ after the surgery is
\begin{equation}
\label{eqn:rot}
\rot(L_i)=\rot_i-\langle\vrot,M^{-1}\vlk_i\rangle.
\end{equation}
\subsubsection{The $d_3$-invariant}
In order to compute the $d_3$-invariant with formula (\ref{eqn:d3surgery}),
we need to read off $c^2,\sigma$ and $\chi$ from the surgery diagram.
Each surgery knot corresponds to the attaching of a $2$-handle,
hence the Euler characteristic of the handlebody is
$\chi=1+n$.

The signature $\sigma$ can be determined as
the signature of the linking matrix~$M$. More topologically,
one can determine $\sigma$ from the Kirby moves for showing that
the surgery diagram actually gives a description of~$S^3$. These Kirby moves
involve handle slides and the blowing down of $(\pm 1)$-framed unknots.
The signature $\sigma$ equals the number of these blow-downs,
counted with sign, minus the number of $(\pm 1)$-framed unknots that may
have been introduced into the Kirby diagram to replace twisting boxes.

The computation of $c^2$ has been explained in~\cite{dgs04}.
Find the solution vector $\bfx$ of the equation $M\bfx=\vrot$;
then $c^2=\bfx^{\ttt}M\bfx=\langle\bfx,\vrot\rangle$.
\subsection{Detecting exceptional links}
In order to decide whether a Legendrian knot presented
in a contact surgery diagram of $(S^3,\xi)$ is exceptional, we first need to
verify that the ambient contact structure $\xi$ is overtwisted.
If the $d_3$-invariant differs from that of the standard
structure~$\xist$, that is, if $d_3(\xi)\neq -\frac{1}{2}$, this
is obvious. In the case where $d_3(\xi)=-\frac{1}{2}$, it suffices to find
a Legendrian knot in the surgered manifold that violates the
Bennequin inequality \cite[Theorem~4.6.36]{geig08} for Legendrian knots
in tight contact $3$-manifolds. In our examples, one of $L_0$ or $L_1$
will have this property.

Secondly, we need to establish that the contact structure on the
link complement $S^3\setminus(L_0\sqcup L_1)$ is tight. The method we use is
to perform contact surgeries on $L_0$ and $L_1$, perhaps also on
Legendrian push-offs of these knots, such that the
resulting contact manifold is tight. If there had been an overtwisted
disc in the complement of $L_0\sqcup L_1$, this would persist
after the surgery.

Here we always employ the cancellation lemma from \cite{dige04},
cf.~\cite[Proposition 6.4.5]{geig08}, which says that a contact
$(-1)$-surgery and a contact $(+1)$-surgery along a Legendrian knot
and its Legendrian push-off, respectively, cancel each other.
Thus, if by contact $(-1)$-surgeries along $L_0$ and $L_1$
we can cancel all contact $(+1)$-surgeries in the surgery diagram,
and thus obtain a Stein fillable and hence tight contact $3$-manifold,
the Legendrian Hopf link will have been exceptional. In fact,
in this case there is also no torsion in the complement, so
the link is strongly exceptional.
\subsection{Exceptional unknots}
The individual components $L_0,L_1$ of an exceptional Hopf link may be
exceptional or loose. Which of the two cases occurs for either
component can be decided by referring to the classification of exceptional
unknots due to Eliashberg and Fraser~\cite{elfr09}, cf.~\cite{geon15},
which we recall here.

\begin{thm}[Eliashberg--Fraser]
\label{thm:EF}
Exceptional unknots can only be realised in the contact structure
$\xi_{1/2}$ on~$S^3$. Up to coarse equivalence, they are classified
by their classical invariants, which can take the values
$(\tb,\rot)=(n,\pm(n-1))$, $n\in \N$.
\end{thm}
\subsection{Legendrian realisations}
\label{subsection:Leg-realise}
We now turn our attention to Legendrian realisations of the
Hopf link in terms of Legendrian surgery diagrams. The invariants
of these realisations are collected in Table~\ref{table:se-invariants}
in Section~\ref{section:compute}.
\subsubsection{Case (b1)}
As we shall explain, Figure~\ref{figure:b1} shows
the $2|\ttt_0-1|$ Legendrian realisations of the
Hopf link with $\ttt_0<0$ and $\ttt_1\geq 2$.
The numbers $k,\ell\in \N_0$ may be chosen
subject to the condition $-\ttt_0=k+\ell$, and $n\in\N_0$ is determined by
$\ttt_1=n+2$. This gives $|\ttt_0-1|$ choices for the pair $(k,\ell)$.
Topologically, the diagram in Figure~\ref{figure:b1} is the
one shown in Figure~\ref{figure:b1kirby}. Hence, for $n$ even
a positive Hopf link is realised when $L_0$ and $L_1$ are both oriented
clockwise or both anticlockwise; for $n$ odd one has to orient
one of the two knots clockwise, the other anticlockwise.
In either case, this gives a factor $2$ in the number
of realisations.

\begin{figure}[h]
\labellist
\small\hair 2pt
\pinlabel $L_1$ [br] at 22 178
\pinlabel $+1$ [l] at 58 167
\pinlabel $n$ [r] at 0 100.5
\pinlabel $k$ [r] at 9 32
\pinlabel $\ell$ [l] at 67 34
\pinlabel $-1$ [l] at 60 138
\pinlabel $-1$ [l] at 60 122
\pinlabel $-1$ [l] at 60 80
\pinlabel $-1$ [l] at 60 63
\pinlabel $L_0$ [tl] at 49 7
\endlabellist
\centering
\includegraphics[scale=1.3]{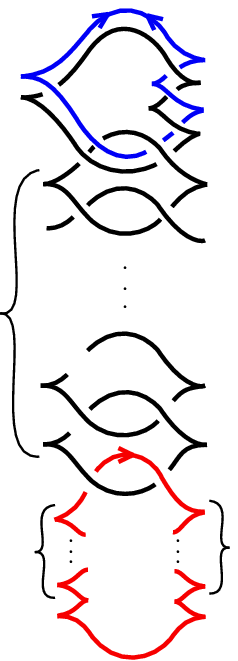}
  \caption{Case (b1): Legendrian Hopf links with $\ttt_0<0$, $\ttt_1\geq 2$.}
  \label{figure:b1}
\end{figure}

The linking matrix $M^{(n)}$ of the surgery diagram, when we order the
$n+1$ surgery knots from top to bottom and orient all of them clockwise,
is the $((n+1)\times (n+1))$-matrix
\[ M^{(n)}=\left(\begin{array}{ccccccc}
-1 & -1 &    &        &        &    &   \\
-1 & -2 & -1 &        &        &    &   \\
   & -1 & -2 & -1     &        &    &   \\
   &    &    & \ddots &        &    &   \\
   &    &    &        & \ddots &    &   \\
   &    &    &        & -1     & -2 & -1\\
   &    &    &        &        & -1 & -2
\end{array}\right) . \]
As shown in \cite[Section~9.1]{geon15}, the determinant of this matrix is
$\det M^{(n)}=(-1)^{n+1}$. The extended linking matrix for $L_0$ is
\[ M_0^{(n)}=\left(\begin{array}{c|cccc}
0      & 0 & \cdots  & 0 & -1 \\ \hline
0      &   &         &   &    \\
\vdots &   & M^{(n)} &   &    \\
0      &   &         &   &    \\
-1     &   &         &   &    
\end{array}\right) . \]
By expanding this matrix along the first row and column, we find
$\det M_0^{(n)}=-\det M^{(n-1)}=(-1)^{n+1}$, too. Hence, by~(\ref{eqn:tb}),
\[ \tb(L_0)=\tb_0+\frac{\det M_0}{\det M}=-(k+\ell+1)+1=-(k+\ell).\]

When $L_0$ is oriented clockwise, we have $\rot_0=\ell-k$.
The first row of $M^{-1}$ is
\[ \bigl(-(n+1),n,-(n-1),\ldots,(-1)^{n+1}\cdot 1\bigr).\]
The vector of rotation numbers of the surgery curves is
$\vrot=(1,0,\ldots,0)^{\ttt}$, and the vector of linking numbers
of $L_0$ with the surgery curves is $\vlk=(0,\ldots,0,-1)^{\ttt}$.
With (\ref{eqn:rot}) this yields
\[ \rot(L_0)=\rot_0-\left\langle
\left(\begin{array}{c}1\\0\\ \vdots\\0\\0\end{array}\right),
M^{-1}
\left(\begin{array}{c}0\\0\\ \vdots\\0\\-1\end{array}\right)
\right\rangle
=\ell-k-(-1)^n.\]

The invariants of $L_1$ have been computed in \cite[Section~9.1]{geon15}.
With the orientation that makes $L_0\sqcup L_1$ a positive Hopf link
(clockwise for $n$ even, counterclockwise for $n$ odd; see
Lemma~\ref{lem:kirby}~(i)), the invariants are
\[ \tb(L_1)=n+2\;\;\;\text{and}\;\;\;\rot(L_1)=-(-1)^n(n+1).\]

The surgery diagram is equivalent to surgery along $n+1$ unlinked
$(-1)$-framed unknots. This gives $\sigma=-(n+1)$ and $\chi=n+2$,
as can be seen from Figure~\ref{figure:b1kirby}.
Moreover, the vector $\bfx=M^{-1}\vrot$ is the first column (or row)
of $M^{-1}$, hence $c^2=\bfx^{\ttt}M\bfx=\langle\bfx,\vrot\rangle=-(n+1)$.
Putting this into formula (\ref{eqn:d3surgery}) gives $d_3(\xi)=1/2$.
The knot $L_1$ is one of the exceptional unknots in $S^3$
described earlier in \cite{elfr09} (see Theorem~\ref{thm:EF})
and, in terms of surgery
diagrams, in~\cite{geon15}; observe that contact $(-1)$-surgery
along $L_1$ cancels the single contact $(+1)$-surgery, and hence
produces a tight contact $3$-manifold. The knot $L_0$ is loose, since
by Theorem~\ref{thm:EF} there
are no exceptional realisations of the unknot with these invariants.
\subsubsection{Case (b2)}
The $|\ttt_0-2|$ Legendrian realisations are shown in
Figure~\ref{figure:b2}, where both $L_0$ and $L_1$ are oriented clockwise,
and $k+\ell=-\ttt_0+1\geq 2$, $k,\ell\in\N_0$.

\begin{figure}[h]
\labellist
\small\hair 2pt
\pinlabel $+1$ [l] at 56 68
\pinlabel $+1$ [l] at 56 60
\pinlabel $L_1$ [r] at 4 52
\pinlabel $k$ [r] at 0 33
\pinlabel $\ell$ [l] at 59 32
\pinlabel $L_0$ [tl] at 41 6
\endlabellist
\centering
\includegraphics[scale=1.7]{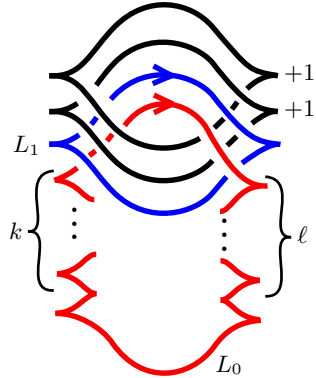}
  \caption{Case (b2): Legendrian Hopf links with $\ttt_0<0$, $\ttt_1=1$.}
  \label{figure:b2}
\end{figure}

The relevant data for the knot $L_1$ are
\[ M=\begin{pmatrix}0&-1\\ -1&0\end{pmatrix}\;\;\;\text{and}\;\;\;
M_0=\begin{pmatrix}0&-1&-1\\-1&0&-1\\-1&-1&0\end{pmatrix}.\]
The signature of $M$ is $\sigma=0$. The vector $\vrot$
is the zero vector. This gives $\tb(L_1)=1$ and $\rot(L_1)=0$.

The knot $L_0$ by itself can be viewed as a stabilisation
of~$L_1$. This yields $\tb(L_0)=1-(k+\ell)=\ttt_0$ and
$\rot(L_0)=\ell-k$. So for case (b2) we want $k+\ell\geq 2$.

The Kirby moves showing that $L_0\sqcup L_1$ is a positive Hopf link
in $S^3$ are shown in Figure~\ref{figure:b2kirby}. The computation of
$d_3=\frac{1}{2}$ for the contact structure on~$S^3$ after the surgery
is straightforward.

Contact $(-1)$-surgery along $L_1$ leaves only a single
contact $(+1)$-surgery along a $\tb=-1$ unknot, which produces
the tight contact structure on $S^1\times S^2$,
see~\cite[Lemma~4.3]{dgs04}. This shows that $L_1$ by itself
is exceptional. (An alternative surgery picture for $L_1$
is shown in \cite[Figure~3]{geon15}.) The knot $L_0$ is loose,
since by Theorem~\ref{thm:EF} there is no exceptional realisation of the
unknot with negative~$\tb$.

The same arguments, with the same surgery picture,
apply when $k+\ell=1$, which gives the subcase $(\ttt_0,\ttt_1)=(0,1)$
of~(d). For $k=\ell=0$ we obtain case~(c1). The only difference
with (b2) is that now $L_0$ is also exceptional, since it
is simply a parallel copy of~$L_1$.
\subsubsection{Case (c2)}
The two realisations with $(\ttt_0,\ttt_1)=(2,1)$ will be described
in Section~\ref{subsubsection:++}. The three realisations with
$(\ttt_0,\ttt_1)=(3,1)$ are shown in Figure~\ref{figure:c2-3-1}.
The Kirby moves in Figure~\ref{figure:c2-3-1kirby} demonstrate
that Figure~\ref{figure:c2-3-1} does indeed depict a Hopf link.

\begin{figure}[h]
\labellist
\small\hair 2pt
\pinlabel $L_0$ [br] at 22 102
\pinlabel $+1$ [l] at 80 91
\pinlabel $+1$ [l] at 77 44
\pinlabel $+1$ [l] at 77 35 
\pinlabel $+1$ [l] at 77 27
\pinlabel $L_1$ [tl] at 64 12
\pinlabel $L_0$ [br] at 160 97
\pinlabel $+1$ [l] at 211 76
\pinlabel $+1$ [l] at 206 41
\pinlabel $+1$ [l] at 206 33
\pinlabel $+1$ [l] at 206 25
\pinlabel $L_1$ [tl] at 195 12
\endlabellist
\centering
\includegraphics[scale=1.4]{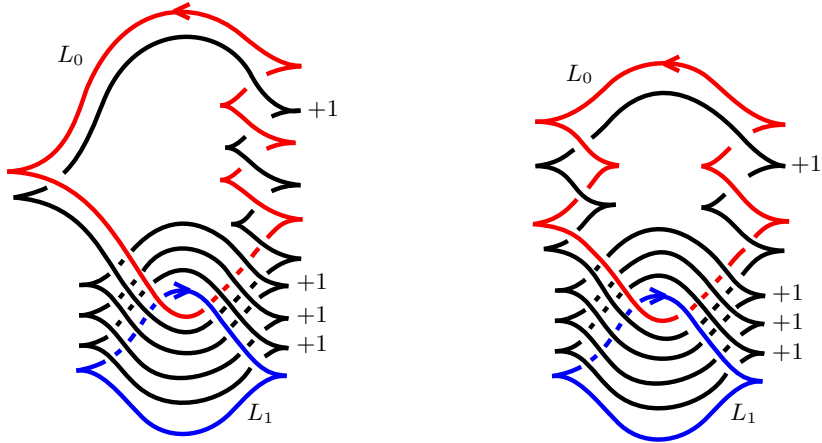}
  \caption{Case (c2): Legendrian Hopf links with $\ttt_0=3$, $\ttt_1=1$.}
  \label{figure:c2-3-1}
\end{figure}

The four realisations with $(\ttt_0,\ttt_1)=(2,2)$ are shown in
Figure~\ref{figure:c2-2-2}. The Kirby moves are similar to the previous case.

\begin{figure}[h]
\labellist
\small\hair 2pt
\pinlabel $L_0$ [br] at 26 117
\pinlabel $+1$ [l] at 80 106
\pinlabel $+1$ [l] at 80 38
\pinlabel $+1$ [l] at 80 29
\pinlabel $L_1$ [tl] at 66 11
\pinlabel $L_0$ [br] at 154 105
\pinlabel $+1$ [l] at 207 95
\pinlabel $+1$ [l] at 207 45
\pinlabel $+1$ [l] at 207 36
\pinlabel $L_1$ [tl] at 193 19
\endlabellist
\centering
\includegraphics[scale=1.2]{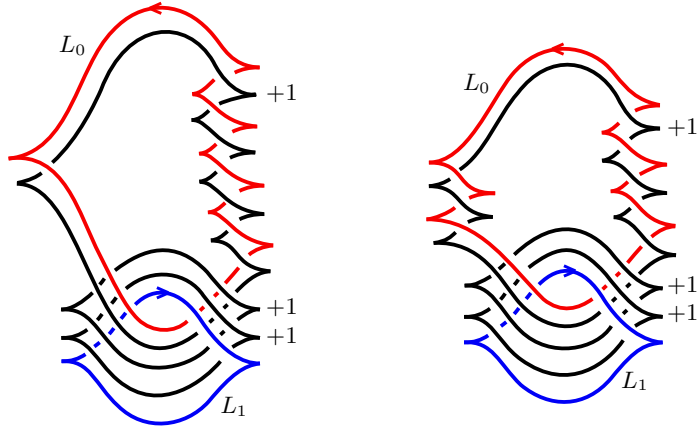}
  \caption{Case (c2): Legendrian Hopf links with $\ttt_0=\ttt_1=2$.}
  \label{figure:c2-2-2}
\end{figure}

The computations of the invariants for the examples in
Figures \ref{figure:c2-3-1} can be found
in Section~\ref{section:compute}; those for Figure~\ref{figure:c2-2-2}
are analogous and are left to the reader.
\subsubsection{Case (c3)}
Figure~\ref{figure:c3-t0-1}, where $n\in\N_0$, shows the four realisations
with $\ttt_0\geq 4$ and $\ttt_1=1$. The computation of the
invariants can be found in Section~\ref{section:compute}.
By Lemma~\ref{lem:kirby}~(ii) we need to give $L_0$ and $L_1$ the same
orientation in the plane for $n$ even; for $n$ odd, the opposite one.

\begin{figure}[h]
\labellist
\small\hair 2pt
\pinlabel $L_0$ [br] at 23 193
\pinlabel $+1$ [l] at 59 178
\pinlabel $-1$ [l] at 59 150
\pinlabel $-1$ [l] at 59 133
\pinlabel $-1$ [l] at 59 92
\pinlabel $-1$ [l] at 59 79
\pinlabel $-1$ [l] at 61 58
\pinlabel $+1$ [l] at 62 37
\pinlabel $+1$ [l] at 62 28
\pinlabel $+1$ [l] at 62 20
\pinlabel $n$ [r] at 0 114
\pinlabel $L_1$ [tl] at 54 10
\pinlabel $L_0$ [br] at 128 193
\pinlabel $+1$ [r] at 117 179
\pinlabel $-1$ [r] at 117 150
\pinlabel $-1$ [r] at 117 137
\pinlabel $-1$ [r] at 117 94
\pinlabel $-1$ [r] at 117 80
\pinlabel $-1$ [l] at 170 60
\pinlabel $+1$ [l] at 172 37
\pinlabel $+1$ [l] at 172 29
\pinlabel $+1$ [l] at 172 21
\pinlabel $n$ [l] at 176 115
\pinlabel $L_1$ [tl] at 159 10
\endlabellist
\centering
\includegraphics[scale=1.5]{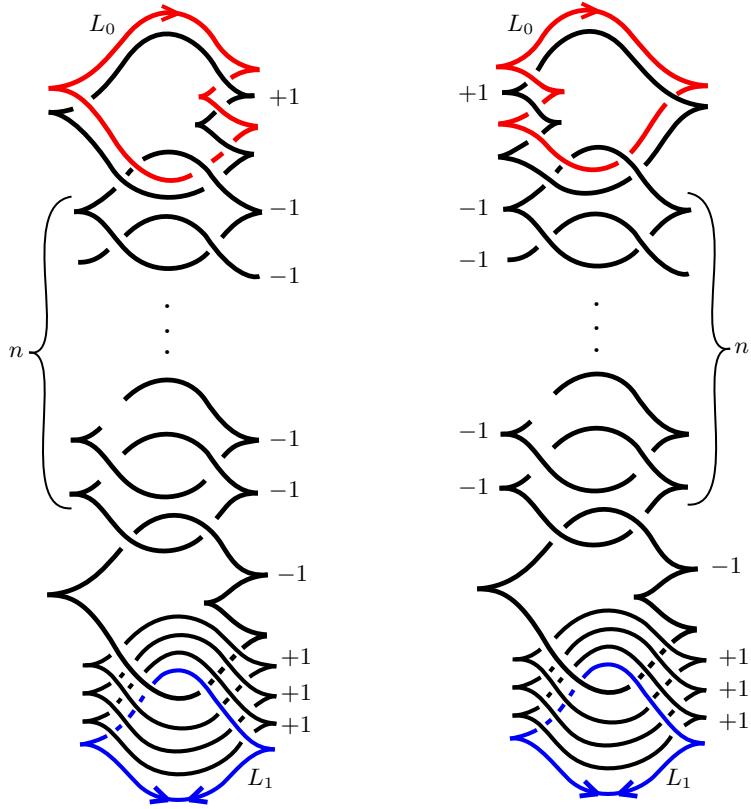}
  \caption{Case (c3): Legendrian Hopf links with $\ttt_0\geq 4$, $\ttt_1=1$.}
  \label{figure:c3-t0-1}
\end{figure}

The six realisations with $\ttt_0\geq 3$ and $\ttt_1=2$ are shown in
Figure~\ref{figure:c3-t0-2}, again with $n\in\N_0$; the invariants are
listed in Table~\ref{table:se-invariants}. For $n$ odd, $L_0$ and $L_1$ are
given the same orientation, for $n$ even, the opposite one.
The computation of the invariants is analogous to the other cases.
We omit the details.

\begin{figure}[h]
\labellist
\small\hair 2pt
\pinlabel $L_1$ [br] at 20 180
\pinlabel $+1$ [l] at 62 167
\pinlabel $-1$ [l] at 62 119
\pinlabel $-1$ [l] at 62 103
\pinlabel $-1$ [l] at 62 62
\pinlabel $-1$ [l] at 62 49
\pinlabel $+1$ [l] at 58 25
\pinlabel $L_0$ [br] at 21 30
\pinlabel $n$ [r] at 1 83
\pinlabel $L_1$ [br] at 119 172
\pinlabel $+1$ [l] at 159 158
\pinlabel $-1$ [l] at 159 119
\pinlabel $-1$ [l] at 159 103
\pinlabel $-1$ [l] at 159 61
\pinlabel $-1$ [l] at 159 48
\pinlabel $+1$ [l] at 154 25
\pinlabel $L_0$ [br] at 118 30
\pinlabel $n$ [r] at 98 83
\pinlabel $L_1$ [bl] at 238 180 
\pinlabel $+1$ [r] at 198 168
\pinlabel $-1$ [r] at 201 121
\pinlabel $-1$ [r] at 201 108
\pinlabel $-1$ [r] at 201 64
\pinlabel $-1$ [r] at 201 49
\pinlabel $+1$ [l] at 247 26
\pinlabel $L_0$ [br] at 209 30
\pinlabel $n$ [l] at 260 84
\endlabellist
\centering
\includegraphics[scale=1.1]{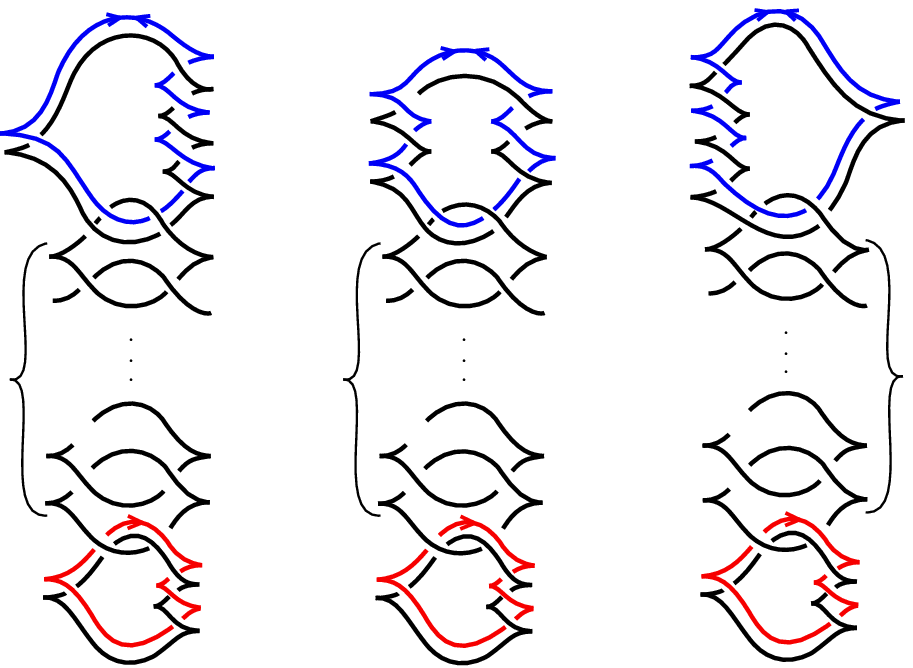}
  \caption{Case (c3): Legendrian Hopf links with $\ttt_0\geq 3$, $\ttt_1=2$.}
  \label{figure:c3-t0-2}
\end{figure}

\subsubsection{Case (c4)}
The eight realisations for each choice of $\ttt_0,\ttt_1\geq 3$
are shown in Figure~\ref{figure:c4}, where $n,m\in\N_0$.
The Kirby moves showing that
these diagrams do indeed depict a Hopf link are similar to those
in Figure~\ref{figure:b1kirby}. For $n+m$ even, one needs to give
$L_0$ and $L_1$ the same orientation, for $n+m$ odd, the
opposite one. By tracking the contact framing of
$L_0$ and $L_1$ through these Kirby moves one finds that
$\ttt_0=n+3$ and $\ttt_1=m+3$. The remaining calculations of the
invariants, listed in Table~\ref{table:se-invariants}, are
analogous to the other cases.

\begin{figure}[h]
\labellist
\small\hair 2pt
\pinlabel $L_0$ [br] at 23 290
\pinlabel $+1$ [l] at 60 279
\pinlabel $-1$ [l] at 60 250
\pinlabel $-1$ [l] at 60 234
\pinlabel $-1$ [l] at 60 192
\pinlabel $-1$ [l] at 60 175
\pinlabel $n$ [r] at 0 213
\pinlabel $-1$ [l] at 60 159
\pinlabel $-1$ [l] at 60 126
\pinlabel $-1$ [l] at 60 110
\pinlabel $-1$ [l] at 60 68
\pinlabel $-1$ [l] at 60 51
\pinlabel $m$ [r] at 0 89
\pinlabel $L_1$ [bl] at 22 16
\pinlabel $+1$ [l] at 60 13
\pinlabel $L_0$ [br] at 124 290
\pinlabel $+1$ [l] at 161 279
\pinlabel $-1$ [l] at 161 250
\pinlabel $-1$ [l] at 161 234
\pinlabel $-1$ [l] at 161 192
\pinlabel $-1$ [l] at 161 175
\pinlabel $n$ [r] at 101 213
\pinlabel $-1$ [l] at 161 152
\pinlabel $-1$ [l] at 161 126
\pinlabel $-1$ [l] at 161 110
\pinlabel $-1$ [l] at 161 68
\pinlabel $-1$ [l] at 161 51
\pinlabel $m$ [r] at 101 89
\pinlabel $L_1$ [bl] at 123 16
\pinlabel $+1$ [l] at 161 13
\pinlabel $L_0$ [bl] at 252 290
\pinlabel $+1$ [r] at 214 279
\pinlabel $-1$ [r] at 214 250
\pinlabel $-1$ [r] at 214 234
\pinlabel $-1$ [r] at 214 192
\pinlabel $-1$ [r] at 214 175
\pinlabel $n$ [l] at 273 213
\pinlabel $-1$ [r] at 210 152
\pinlabel $-1$ [r] at 214 126
\pinlabel $-1$ [r] at 214 110
\pinlabel $-1$ [r] at 214 68
\pinlabel $-1$ [r] at 214 51
\pinlabel $m$ [l] at 273 89
\pinlabel $L_1$ [bl] at 225 16
\pinlabel $+1$ [l] at 260 13
\pinlabel $L_0$ [bl] at 353 290
\pinlabel $+1$ [r] at 315 279
\pinlabel $-1$ [r] at 315 250
\pinlabel $-1$ [r] at 315 234
\pinlabel $-1$ [r] at 315 192
\pinlabel $-1$ [r] at 315 175
\pinlabel $n$ [l] at 374 213
\pinlabel $-1$ [r] at 319 158
\pinlabel $-1$ [r] at 315 126
\pinlabel $-1$ [r] at 315 110
\pinlabel $-1$ [r] at 315 68
\pinlabel $-1$ [r] at 315 51
\pinlabel $m$ [l] at 374 89
\pinlabel $L_1$ [bl] at 326 16
\pinlabel $+1$ [l] at 361 13
\endlabellist
\centering
\includegraphics[scale=.85]{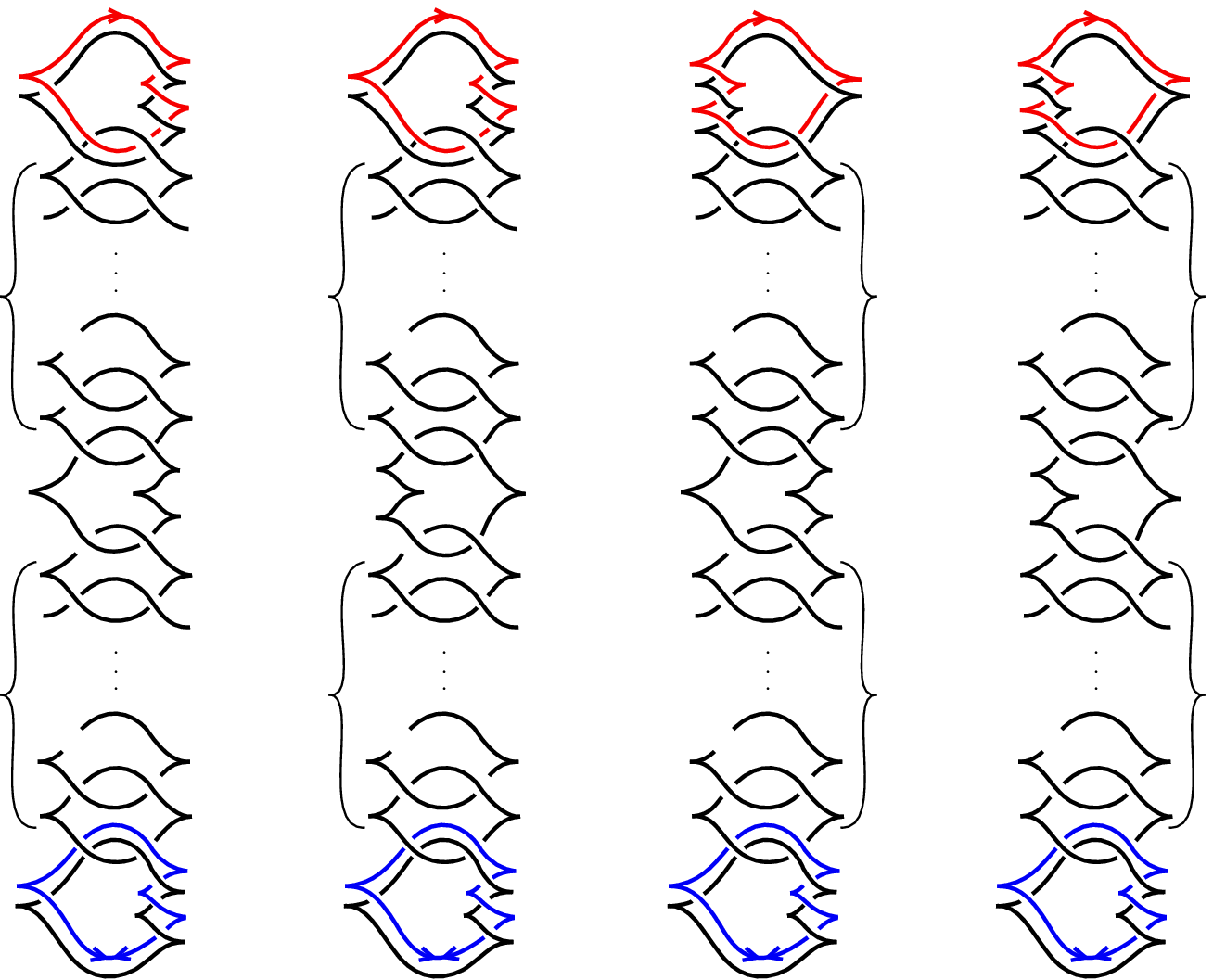}
  \caption{Case (c4): Legendrian Hopf links with $\ttt_0\geq\ttt_1\geq 3$.}
  \label{figure:c4}
\end{figure}

\begin{rem}
Beware that there is no direct correspondence between the four diagrams
in Figure~\ref{figure:c4} and the four subcases of (c4) listed
in Table~\ref{table:se-invariants}. Rather, the correspondence hinges on
the parity of $m$ and $n$. For instance, the realisations with
$d_3=-\frac{1}{2}$ are given by the first diagram (from the left)
for $m,n$ even, by the second for $n$ even and $m$ odd, the third for
$n$ odd and $m$ even, and the fourth if both $n$ and $m$ are odd.
The same caveat applies in the case~(c3).
\end{rem}

\subsubsection{Case (d)}
This is the case with $\ttt_0=0$. Exceptional Legendrian realisations
of the Hopf link with $\ttt_1\leq 0$
are shown in Figure~\ref{figure:d}, where $n\geq 2$. For the computation of
the invariants see Section~\ref{section:compute}. The link components
are loose by Theorem~\ref{thm:EF}. Examples with
$\ttt_1\geq 2$ are realised by Figure~\ref{figure:b1}
with $k=\ell=0$; all the arguments from case (b1) apply likewise
for this choice of $k,\ell$. Figure~\ref{figure:b2} from case (b2)
with $(k,\ell)\in\{(1,0),(0,1)\}$ gives the realisations with $\ttt_1=1$.

\begin{figure}[h]
\labellist
\small\hair 2pt
\pinlabel $L_1$ [br] at 26 136
\pinlabel $L_0$ [br] at 39 101
\pinlabel $+1$ [l] at 94 112
\pinlabel $+1$ [l] at 93 65
\pinlabel $+1$ [l] at 93 47
\pinlabel $+1$ [l] at 93 39
\pinlabel $+1$ [l] at 93 30
\pinlabel $n$ [r] at 0 49
\endlabellist
\centering
\includegraphics[scale=1.1]{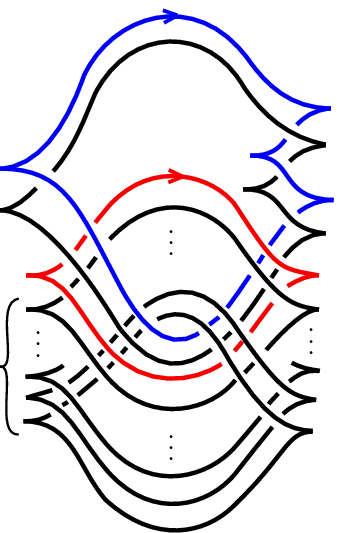}
  \caption{Case (d): Legendrian Hopf links with $\ttt_0=0$, $\ttt_1\leq 0$.}
  \label{figure:d}
\end{figure}
\subsubsection{Summary}
In Table~\ref{table:summary} we arrange the strongly exceptional cases
from Table~\ref{table:se-invariants} with $\ttt_0,\ttt_1\in\N$
a little more systematically.
In the second line there also belongs the case with $(\ttt_i,\ttr_i)=
(2,\pm 1)$.

\begin{table}
\begin{tabular}{|c|c|c|c|c|} \hline
$\ttt_0$ & $\ttr_0$        & $\ttt_1$ & $\ttr_1$        & $d_3$   \\ \hline
                                                                     \hline
$\geq 1$ & $\pm(\ttt_0+1)$ & $\geq 1$ & $\pm(\ttt_1+1)$ & $-1/2$  \\ \hline
$\geq 3$ & $\pm(\ttt_0-3)$ & $\geq 1$ & $\mp(\ttt_1-1)$ & $3/2$   \\ \hline
$\geq 3$ & $\pm(\ttt_0-1)$ & $\geq 2$ & $\mp(\ttt_1-3)$ & $3/2$   \\ \hline
$\geq 3$ & $\pm(\ttt_0-1)$ & $\geq 3$ & $\pm(\ttt_1-1)$ & $3/2$   \\ \hline
\end{tabular}
\label{table:summary}
\vspace{1.5mm}
\caption{Strongly exceptional realisations with $\ttt_0,\ttt_1\in\N$.}
\end{table}
\section{Contact structures on $S^3$ as contact cuts}
\label{section:cut}
In order to classify Legendrian Hopf links in $S^3$ where the contact
structure on the complement has non-zero twisting, it is useful
to describe certain contact structures on $S^3$ as
contact cuts in the sense of Lerman~\cite{lerm01}.
This construction (without a reference to Lerman) was used
in \cite{dyma04} to describe Legendrian knots and links
in~$S^3$; see also~\cite[Section~2.5]{voge18}. We provide
additional details and simplify the computation of
the classical invariants in these models of~$S^3$.

On $T^2\times[0,1]$ with coordinates $x,y\in\R/\Z$ and $z\in [0,1]$
we consider the contact forms
\begin{equation}
\label{eqn:alpha_p}
\alpha_p=\sin\Bigl(\frac{2p+1}{2}\pi z\Bigr)\,\rmd x+
\cos\Bigl(\frac{2p+1}{2}\pi z\Bigr)\,\rmd y,\;\;\; p\in\N_0.
\end{equation}
These contact forms are invariant under the $S^1$-actions
generated by $\partial_x$ and $\partial_y$. Along $\{z=0\}$,
the vector field $\partial_x$ is contained in $\ker\alpha_p$;
along $\{z=1\}$, we have $\partial_y\in\ker\alpha_p$.
By the contact cut construction \cite[Proposition~2.15]{lerm01},
the $1$-form $\alpha_p$ descends to a contact form on
the quotient manifold obtained by
collapsing the first $S^1$-factor in $S^1\times S^1\times\{0\}$,
and the second factor in $S^1\times S^1\times\{1\}$.
This quotient manifold $T^2\times[0,1]/\!\sim$ is easily seen to be~$S^3$,
since collapsing the circles in question is topologically equivalent
to attaching a copy
of a solid torus $S^1\times D^2$ to each boundary component
of $T^2\times [0,1]$, where the meridian $\{*\}\times\partial D^2$
is sent to the first or the second $S^1$-factor of~$T^2$, respectively.

Our aim in this section is to prove the following proposition.
Recall that $\xi_d$ denotes the overtwisted contact structure
on $S^3$ with $d_3(\xi_d)=d$.

\begin{prop}
\label{prop:cut}
The contact structure on $S^3=T^2\times[0,1]/\!\sim$ induced by the
contact form $\alpha_p$ equals
\begin{itemize}
\item[(i)] $\xist$ for $p=0$;
\item[(ii)] $\xi_{1/2}$ for $p$ odd;
\item[(iii)] $\xi_{-1/2}$ for $p\geq 2$ even.
\end{itemize}
\end{prop}

Parts of the argument for proving this statement are contained in
Examples 2.16, 2.19 and the proof of Theorem~3.1 in~\cite{lerm01}.

\begin{rem}
\label{rem:tight}
On $T^2\times[0,1]$, the contact structures $\ker\alpha_p$ are tight,
since they embed into the standard tight contact structure on~$\R^3$,
cf.~\cite[Corollary~6.5.10]{geig08}.
\end{rem}
\subsection{A transverse Hopf link}
We begin with a topological preparation.
The circles
\[ C_0:=\Bigl\{\frac{1}{2}\Bigr\}\times S^1\times\{0\}\]
and
\[ C_1:=S^1\times\Bigl\{\frac{1}{2}\Bigr\}\times\{1\}\]
in $T^2\times [0,1]$ descend to unknots in
$S^3=T^2\times[0,1]/\!\sim$, which we continue to denote by $C_0,C_1$.
See Figure~\ref{figure:box-hopf-link}, where
$T^2$ is illustrated as $(\R/\Z)^2$. On the back face of the cube
$[0,1]^3$ the horizontal segments are collapsed; on the front face,
the vertical ones.

\begin{figure}[h]
\labellist
\small\hair 2pt
\pinlabel $x$ [t] at 320 47
\pinlabel $y$ [r] at 146 223
\pinlabel $z$ [tl] at 11 2
\pinlabel $C_0$ [l] at 221 143
\pinlabel $C_1$ [b] at 129 87
\endlabellist
\centering
\includegraphics[scale=.6]{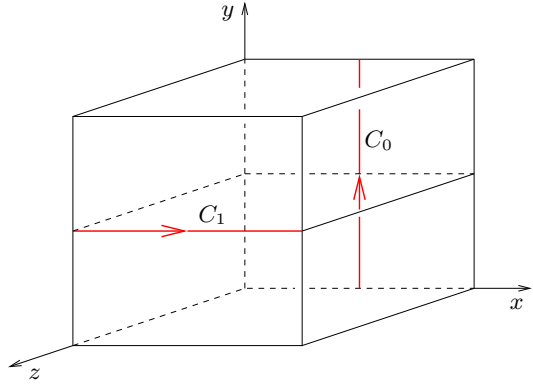}
  \caption{A transverse Hopf link in $(S^3,\ker\alpha_p)$.}
  \label{figure:box-hopf-link}
\end{figure}

\begin{lem}
The unknots $C_0,C_1$ form a positive Hopf link in~$S^3$.
The unknot $C_0$ is positively transverse to the contact structure
$\ker\alpha_p$ on~$S^3$; the unknot $C_1$ is positively or negatively
transverse to $\ker\alpha_p$, depending on $p$ being even or odd.
\end{lem}

\begin{proof}
The horizontal square $[0,1]\times\bigl\{\frac{1}{2}\bigr\}\times[0,1]$ in
Figure~\ref{figure:box-hopf-link} descends to a $2$-disc
in $S^3$ with boundary~$C_1$, and the unknot $C_0$ intersects
this disc positively in a single point. The transversality property
of $C_0,C_1$ is obvious from the definition of~$\alpha_p$.
\end{proof}
\subsection{The standard contact structure}
Consider the $8$-dimensional manifold
\[ W=T^2\times\C\times\C^2\]
with $S^1$-valued coordinates $q_1,q_2$ on $T^2=(\R/2\pi\Z)^2$,
Cartesian coordinates $p_1+\rmi p_2$
on $\C\setminus\{0\}$, and polar coordinates $(r_i,\varphi_i)$
on the last two $\C$-factors. Define a $1$-form
\[ \lambda=p_1\,\rmd q_1+p_2\,\rmd q_2+\frac{1}{2}r_1^2\,\rmd\varphi_1
+r_2^2\,\rmd\varphi_2\]
on $W$. Notice that the first two summands of $\alpha$ define
the canonical Liouville $1$-form on $T^2\times\C$,
regarded as the unit cotangent bundle of $T^2$.
The $2$-form $\omega=\rmd\lambda$ is a symplectic form on~$W$.

The vector field
\[ Y=p_1\partial_{p_1}+p_2\partial_{p_2}+\frac{1}{2}r_1\partial_{r_1}
+\frac{1}{2}r_2\partial_{r_2}\]
is a Liouville vector field for~$\omega$, i.e.\ $L_Y\omega=\omega$.
The vector fields
\[ X_i=\partial_{q_i}-\partial_{\varphi_i},\;\;\; i=1,2,\]
are Hamiltonian vector fields for~$\omega$, corresponding to the
Hamiltonian functions $p_i-r_i^2/2$, $i=1,2$. The three
vector fields $X_1,X_2,Y$ commute pairwise with each other.

The Hamiltonian $T^2$-action generated by $(X_1,X_2)$ has the
momentum map
\[ \mu(q_1,q_2,p_1,p_2,r_1,\varphi_1,r_2,\varphi_2)=
\Bigl(p_1-\frac{1}{2}r_1^2,\,p_2-\frac{1}{2}r_2^2\Bigr).\]
The level set $\mu^{-1}(0,0)$ is regular, hence a $6$-dimensional manifold,
and the induced $T^2$-action on this manifold is free. A transversal
to the $T^2$-action on $\mu^{-1}(0,0)$ is defined by
\[ \bigl\{ q_1=\varphi_1,\, q_2=\varphi_2\}\subset\mu^{-1}(0,0).\]
This shows that the reduced manifold $\mu^{-1}(0,0)/T^2$
is the symplectic manifold
\[ \bigl(\C^2,r_1\,\rmd r_1\wedge\rmd\varphi_1+
r_2\,\rmd r_2\wedge\rmd\varphi_2\bigr).\]
The Liouville vector field $Y$ is tangent to the level set $\mu^{-1}(0,0)$
and descends as a Liouville vector field $\overline{Y}$ to the symplectic
quotient~$\C^2$. On the hypersurface
\[ \Bigl\{ \frac{1}{4}(r_1^4+r_2^4)=1\Bigl\}\subset\C^2, \]
which is a diffeomorphic copy of~$S^3$ transverse to~$\overline{Y}$,
this Liouville vector field induces a contact form defining~$\xist$.

The $\C$-component $-\partial_{\varphi_i}$ of $X_i$ defines
a free $S^1$-action, except at the origin $r_i=0$. It follows that
one equally obtains the contact structure $\xist$ on $S^3$ 
from the contact manifold
\[ \bigl\{ p_1^2+p_2^2=1,\, p_1,p_2\geq 0\bigr\}\subset T^2\times\C,\] 
with contact form $p_1\,\rmd q_1+p_2\,\rmd q_2$, by taking the
quotient of the boundary components $\{p_1=0\}$, $\{p_2=0\}$ under
the $S^1$-action defined by $\partial_{q_1}$
and~$\partial_{q_2}$, respectively. This is exactly the
description of the contact structure on $S^3$ as a quotient
of $\bigl(T^2\times[0,1],\alpha_0\bigr)$ given before
Proposition~\ref{prop:cut}, and hence proves part (i) of that
proposition.
\subsection{A global frame}
In order to determine the contact structures $\ker\alpha_p$
on $S^3$ for $p\geq 1$, and for computing the classical invariants
of Legendrian and transverse knots in these contact structures,
we now exhibit a global frame for the contact planes
$\ker\alpha_p$ on~$S^3$.

A global frame for $\ker\alpha_p$ on $T^2\times[0,1]$ is given by
\begin{equation}
\label{eqn:frame0}
\partial_z\;\;\;\text{and}\;\;\;
X_p:=\cos\Bigl(\frac{2p+1}{2}\pi z\Bigr)\,\partial_x-
\sin\Bigl(\frac{2p+1}{2}\pi z\Bigr)\,\partial_y.
\end{equation}
This frame does not, however, descend to~$S^3$.
The differential of $\alpha_p$ is
\[ \rmd\alpha_p=\frac{2p+1}{2}\pi\cos\Bigl(\frac{2p+1}{2}\pi z\Bigr)\,
\rmd z\wedge\rmd x-
\frac{2p+1}{2}\pi\sin\Bigl(\frac{2p+1}{2}\pi z\Bigr)\;\rmd z\wedge\rmd y,\]
so the frame $\partial_z, X_p$ is positive for the orientation of
$\ker\alpha_p$ defined by this $2$-form. Whenever we speak of a frame
for $\ker\alpha_p$, this orientation assumption will be understood.

Remember that at $z=0$, where $X_p=\partial_x$,
we collapse the $x$-circles in~$T^2$.
Figure~\ref{figure:frame-cut} shows such a circle, with the
$z$-direction pointing to the exterior. Performing the cut
is topologically the same as filling in the circle with a disc.
So the frame that extends is the one that does not rotate as one
goes along the circle, for instance
\begin{equation}
\label{eqn:frame1}
\cos(2\pi x)\,\partial_z-\sin(2\pi x)\,X_p\;\;\;\text{and}\;\;\;
\sin(2\pi x)\,\partial_z+\cos(2\pi x)\, X_p.
\end{equation}

\begin{figure}[h]
\labellist
\small\hair 2pt
\pinlabel $\partial_z$ [t] at 452 250
\pinlabel $\partial_x$ [l] at 401 310
\pinlabel $\partial_z$ [tl] at 397 393 
\pinlabel $\partial_x$ [bl] at 318 399
\endlabellist
\centering
\includegraphics[scale=.4]{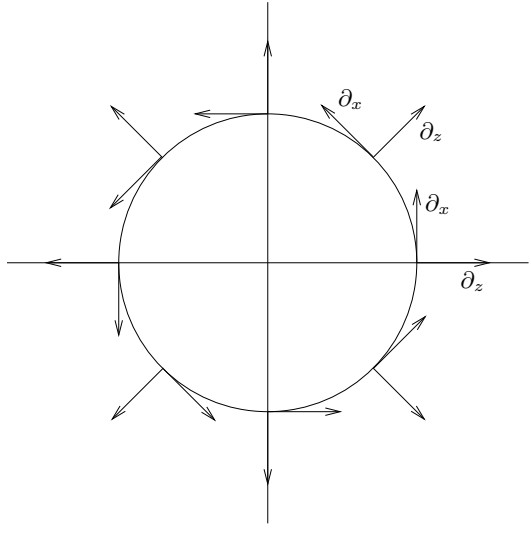}
  \caption{Finding the frame that descends to the cut manifold.}
  \label{figure:frame-cut}
\end{figure}

At $z=1$ we collapse the $y$-circles. Here $X_p$ equals
$-\partial_y$ or $\partial_y$, depending on $p$ being even or odd.
Arguing as above, while observing that $\partial_z$ now points
into the $2$-disc over which we want to extend the frame,
we find that we may work with the frame
\begin{equation}
\label{eqn:frame2}
\cos(2\pi y)\,\partial_z-(-1)^p\sin(2\pi y)\,X_p\;\;\;\text{and}\;\;\;
(-1)^p\sin(2\pi y)\,\partial_z+\cos(2\pi y)\, X_p.
\end{equation}

The extendability of the frames (\ref{eqn:frame1}) and (\ref{eqn:frame2})
is not affected if we rotate the first one along the
$y$-direction, or the second one along the $x$-direction. This proves
the following lemma.

\begin{lem}
\label{lem:frame}
A global positive frame for the contact structure $\ker\alpha_p$
on $S^3$ is given by
\[ X_p^1:=\cos\bigl(2\pi(x+(-1)^py)\bigr)\,\partial_z-
\sin\bigl(2\pi(x+(-1)^py)\bigr)\,X_p\]
and
\[ X_p^2:=\sin\bigl(2\pi(x+(-1)^py)\bigr)\,\partial_z+
\cos\bigl(2\pi(x+(-1)^py)\bigr)\,X_p,\]
where $X_p$ is defined in~(\ref{eqn:frame0}).
\qed
\end{lem}
\subsection{The $d_3$-invariant of $\ker\alpha_p$}
There are two ways to compute $d_3(\ker\alpha_p)$. One can
directly compute the Hopf invariant of $\ker\alpha_p$, or one
can interpret $\ker\alpha_p$ as the contact structure obtained
via a suitable $\pi$-Lutz twist from $\ker\alpha_{p-1}$.

For the first approach, we observe that the Reeb vector field
of $\alpha_p$ on $T^2\times[0,1]$ is
\[ R_p=\sin\Bigl(\frac{2p+1}{2}\pi z\Bigr)\,\partial_x+
\cos\Bigl(\frac{2p+1}{2}\pi z\Bigr)\,\partial_y.\]
This vector field descends as the Reeb vector field
of the contact form on $S^3$ obtained via the cut construction.

As in Section~\ref{section:contS3} we normalise the Hopf invariant $h$
such that $h(\ker\alpha_0)=h(\xist)=0$, in other words, we compute the
Hopf invariant of the Gau{\ss} map of a tangent $2$-plane field
with respect to the trivialisation of the tangent bundle $TS^3$
defined by $R_0,X_0^1,X_0^2$.

The Gau{\ss} map of $\ker\alpha_p$ is the map $S^3\rightarrow S^2$
given by expressing $R_p$ in terms of the frame $R_0,X_0^1,X_0^2$.
A straightforward computation yields
\begin{equation}
\label{eqn:Rp}
R_p=\cos(p\pi z)R_0+\sin(p\pi z)\Bigl(\cos\bigl(2\pi (x+y)\bigr)X_0^2-
\sin\bigl(2\pi(x+y)\bigr)X_0^1\Bigr).
\end{equation}
The preimage $R_1^{-1}(s)$ of a regular value $s\in S^2$
is a collection of circles, oriented in such a way that
the transverse orientation coincides with the orientation of~$S^2$.

\begin{lem}
The values $\pm X_2^0\in S^2$ are regular for the map $R_1\co S^3\rightarrow
S^2$. The preimage $R_1^{-1}(X_2^0)$ is the oriented circle
\[ \gamma_+\co t\longmapsto (x,y,z)=\Bigl(1-t,t,\frac{1}{2}\Bigr);\]
the preimage $R_1^{-1}(-X_2^0)$ equals
\[ \gamma_-\co t\longmapsto (x,y,z)=\Bigl(\frac{1}{2}-t,t,\frac{1}{2}\Bigr).\]
\end{lem}

\begin{proof}
From (\ref{eqn:Rp}) it is clear that the preimages
$R_1^{-1}(\pm X_2^0)$, as sets, are the described circles.
To find the orientation of these circles, one considers the behaviour
of $R_1$ along a small meridional circle of~$\gamma_{\pm}$. For~$\gamma_+$,
at a point on this meridional circle where $z<1/2$, the vector field
$R_1$ has a positive $R_0$-component; for $z>1/2$ a negative one.
For $z=1/2$ and $x+y$ a little larger (resp.\ smaller) than~$1$,
the vector field $R_1$ has a negative (resp.\ positive) $X_0^1$-component.
Since the transverse orientation at $X_2^0\in S^2$ is given by
the oriented frame $R_0,X_0^1$, this gives the claimed orientation
of~$\gamma_+$. For $\gamma_-$ the argument is analogous.
\end{proof}

\begin{proof}[First proof of Proposition~\ref{prop:cut}]
By pushing $\gamma_-$ to $z=0$ and $\gamma_+$ to $z=1$, we see that
the oriented link $\gamma_-\sqcup\gamma_+$ is isotopic
to~$C_0\sqcup -C_1$, hence
\[ h(\ker\alpha_1)=\lk(\gamma_-,\gamma_+)=\lk(C_0,-C_1)=-1.\]

For the $\alpha_p$ with $p>1$ one can argue similarly, but see also
the alternative proof below.
\end{proof}

For the second proof of Proposition~\ref{prop:cut} we compute the
self-linking number of the transverse unknots $C_0,C_1$.

\begin{lem}
\label{lem:slk}
The self-linking number of the transverse unknot $C_i$, $i=0,1$,
in $(S^3,\ker\alpha_p)$ is $\slk_{\ker\alpha_p}(C_i)=(-1)^{p+1}$.
\end{lem}

\begin{proof}
The self-linking number $\slk(K)$ of a homologically trivial transverse
knot $K$ is computed as the linking number $\lk (K,K')$  with a push-off $K'$
in the direction of a non-vanishing section of the contact structure
(over a Seifert surface of~$K$ if the contact structure is not
globally trivial as a $2$-plane bundle).

Along $C_0$ in $T^2\times[0,1]$ we have
\[ X_p^1=\cos\bigl(\pi+(-1)^p2\pi y\bigr)\,\partial_z-
\sin\bigl(\pi+(-1)^p2\pi y\bigr)\,\partial_x.\]
From this one finds the claimed self-linking number as the
intersection number of the push-off $C_0'$ with the
disc $\{1/2\}\times[0,1]\times[0,1]/\!\sim$. One may push $C_0$
a little into the $z$-direction if one worries about $C_0$
sitting in the boundary of $T^2\times[0,1]$ that is being partially
collapsed.

Along $C_1$ we have
\[ X_p^1=\cos\bigl((-1)^p 2\pi x+\pi\bigr)\,\partial_z+
\sin\bigl((-1)^p 2\pi x+\pi\bigr)\,\partial_y,\]
and one then computes similarly.
\end{proof}

\begin{proof}[Second proof of Proposition~\ref{prop:cut}]
We already know that $\ker\alpha_0=\xist$, with Hopf
invariant $h(\ker\alpha_0)=0$.
The contact structure $\ker\alpha_{p+1}$ on $S^3$ is obtained
from $\ker\alpha_p$ by a $\pi$-Lutz twist along~$C_1$.
As in the proof of Lemma~\ref{lem:h-d3} this yields
\[ h(\ker\alpha_{p+1})=h(\ker\alpha_p)+\slk_{\ker\alpha_p}(C_1)=
h(\ker\alpha_p)+(-1)^p.\]
With Lemma~\ref{lem:slk} we find $h(\ker\alpha_p)=0$ for $p$ even,
and $h(\ker\alpha_p)=-1$ for $p$ odd. Proposition~\ref{prop:cut}
now follows with Lemma~\ref{lem:h-d3}.
\end{proof}
\section{Hopf links with twisting complement}
\label{section:twisting}
We now describe Legendrian Hopf links in $S^3$ equipped with
one of the contact structures $\ker\alpha_p$ from Section~\ref{section:cut}.
As we shall see, any Hopf link whose complement
has positive twisting can be accounted for in this way.
\subsection{Examples of torus knots}
Recall the definition of the contact form $\alpha_p$ on
$T^2\times [0,1]$ given in~(\ref{eqn:alpha_p}).
On the torus $T^2\times\{z\}$ this induces a non-singular
linear characteristic foliation
$\bigl(T^2\times\{z\}\bigr)_{\ker\alpha_p}$
of slope $-\tan\bigl((2p+1)\pi z/2\bigr)$.
We choose $z\in[0,1]$ such that this slope is a rational
number $b/a$ (including~$\infty$), with $a, b$ coprime.
We then have two Legendrian knots
\[ \beta_{\pm}^{a,b}\co t\longmapsto (\pm at,\pm bt,z),\;\;\; t\in [0,1],\]
differing only in orientation. Regarded as knots in~$S^3$,
these are $(a,b)$-torus knots.

\begin{lem}
The classical invariants of $\beta_{\pm}^{a,b}$ in $(S^3,\ker\alpha_p)$ are
$\tb(\beta_{\pm}^{a,b})=ab$ and
$\rot(\beta_{\pm}^{a,b})=\pm\bigl(a+(-1)^p b\bigr)$.
\end{lem}

\begin{proof}
As a push-off of $\beta_{\pm}^{a,b}$ in a direction transverse to
$\ker\alpha_p$ we can simply take a parallel copy on $T^2\times\{z\}$. The
Thurston--Bennequin invariant is the linking number of $\beta_{\pm}^{a,b}$
with this push-off, which equals~$ab$.

For the rotation number of $\beta_+^{a,b}$ we need to count the number
of rotations of $X_p$ (which is tangent to~$\beta_+^{a,b}$) relative
to the frame $X_p^1,X_p^2$ as we traverse $\beta_+^{a,b}$ once in
positive direction. Equivalently, we need to count negatively the
number of rotations of $X_p^1$ relative to the frame $\partial_z,X_p$.
Along $\beta_+^{a,b}$ we have, with Lemma~\ref{lem:frame},
\[ X_p^1\bigl(\beta_+^{a,b}(t)\bigr)
=\cos\bigl(2\pi(a+(-1)^pb)t\bigr)\,\partial_z-
\sin\bigl(2\pi(a+(-1)^pb)t\bigr)\,X_p\bigl(\beta_+^{a,b}(t)\bigr).\]
This yields the claimed result $\rot(\beta_{\pm}^{a,b})=
\pm\bigl(a+(-1)^pb\bigr)$.
\end{proof}

These examples were first described by Dymara~\cite{dyma04},
who computed their classical invariants from a generalised
front projection.
\subsection{Examples of Hopf links}
Given $p\in\N_0$ and $a\in\Z$, let $z_a\in (0,\frac{2}{2p+1})$
be the unique number such that
\[ -\tan\Bigl(\frac{2p+1}{2}\pi z_a\Bigr)=\frac{1}{a}.\]
If $p=0$ we assume in addition that $a<0$, so that $z_a\in (0,1)$
in this case, too.
The Legendrian knot $\beta_0:=\beta_+^{a,1}$, corresponding to
this choice of~$z$, i.e.\ $\beta_0(t)=(at,t,z_a)$, is isotopic
to~$C_0$.

Observe that the slope of the characteristic foliation
$\bigl(T^2\times\{z\}\bigr)_{\ker\alpha_p}$ decreases monotonically
from $0$ to $-\infty$ and then from $\infty$ to $0$ as $z$ goes
from $0$ to $\frac{2}{2p+1}$. Hence, any parallel copy
$\beta_{0,z}\co t\mapsto (at,t,z)$ of $\beta_0$ with $0<z<z_a$ is
positively transverse to~$\ker\alpha_p$, see Figure~\ref{figure:slopes}.
For $z=0$, this curve coincides with $C_0$ in~$S^3$. This means that any
such parallel curve constitutes the positive transverse push-off of $\beta_0$
and is transversely isotopic to~$C_0$.

\begin{figure}[h]
\labellist
\small\hair 2pt
\pinlabel $a=-3$ [t] at 109 2
\pinlabel $a=3$ [t] at 396 2
\endlabellist
\centering
\includegraphics[scale=.5]{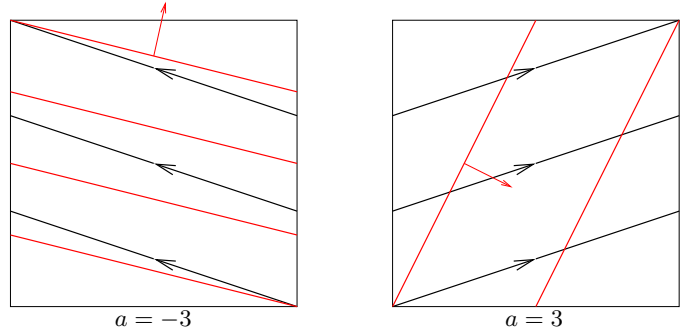}
  \caption{$\beta_{0,z}$ is positively transverse to $\ker\alpha_p$
           for $z<z_a$.}
  \label{figure:slopes}
\end{figure}

In a completely analogous fashion, given $b\in\Z$, let $z_b\in
(1-\frac{2}{2p+1},1)$ be the unique number such that
\[ -\tan\Bigl(\frac{2p+1}{2}\pi z_b\Bigr)=b.\]
Again, for $p=0$ we have to assume $b<0$.
The Legendrian knot $\beta_1:=\beta_+^{1,b}$ is isotopic to~$C_1$.
Any parallel copy $t\mapsto(t,bt,z)$ with $z_b<z<1$ is
positively resp.\ negatively transverse to $\ker\alpha_p$,
depending on $p$ being even or odd, and transversely
isotopic to~$C_1$.

\begin{figure}[h]
\labellist
\small\hair 2pt
\pinlabel $b/1$ [tl] at 264 405
\pinlabel $1/a$ [br] at 27 154
\pinlabel $x$ [t] at 425 214
\pinlabel $y$ [r] at 214 425
\endlabellist
\centering
\includegraphics[scale=.45]{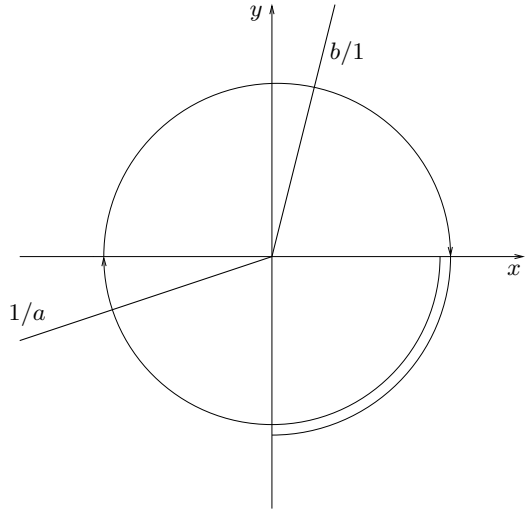}
  \caption{Realising $\beta_0\sqcup \beta_1$ as a Legendrian Hopf link
           in~$\ker\alpha_p$.}
  \label{figure:slopes-hopf}
\end{figure}

When $z_a<z_b$, the link $\beta_0\sqcup\beta_1$ constitutes
a positive Hopf link in $(S^3,\ker\alpha_p)$.
Figure~\ref{figure:slopes-hopf} illustrates that for
$a,b>0$ we need to have $p\geq 2$ to achieve $z_a<z_b$.
For $a,b<0$ we can take any $p\in\N_0$. In all other cases,
any $p\geq 1$ will do.

\begin{prop}
The Hopf link $\beta_0\sqcup\beta_1$ in $(S^3,\ker\alpha_p)$
is exceptional.
\end{prop}

\begin{proof}
Suppose there were an
overtwisted disc in $S^3\setminus(\beta_0\sqcup\beta_1)$.
This disc would persist in the complement of the
transverse push-off $\beta_0^T\sqcup\beta_1^T$ of $\beta_0\sqcup\beta_1$,
since $\beta_0^T\sqcup\beta_1^T$
can be chosen as close to $\beta_0\sqcup\beta_1$ as we like.
But $\beta_0^T\sqcup\beta_1^T$ is transversely
and hence, by \cite[Theorem~2.6.12]{geig08}, ambiently contact
isotopic to $C_0\sqcup C_1$, whose complement is tight
by Remark~\ref{rem:tight}.
\end{proof}

The case $a=b=0$ of this result was proved by
Dymara~\cite[Proposition~4.9]{dyma04}, using a more complicated argument.
Dymara also showed that $\beta_{\pm}^{a,b}$
in $(S^3,\ker\alpha_1)=(S^3,\xi_{1/2})$ is exceptional
for $ab>0$. We can give a very simple proof of this fact, similar
to the one above, if one of $a,b$ equals~$1$.

\begin{prop}
The unknots $\beta_{\pm}^{a,1}$ and $\beta_{\pm}^{1,b}$ with $a,b>0$
in $(S^3,\ker\alpha_1)$ are exceptional.
\end{prop}

\begin{proof}
As illustrated in \cite[Figure~6]{voge18}, an $(a-1)$-fold
negative stabilisation of $\beta_+^{a,1}$ yields
$\beta_+^{1,1}$; the other cases are analogous.

Suppose $\beta_+^{a,1}$ were loose. Then so would be its stabilisation
$\beta:=\beta_+^{1,1}$. The overtwisted disc would persist in the
complement of two parallel copies $\beta_{z_0}$, $\beta_{z_1}$ of
$\beta_+^{1,1}$, with $z_0<\frac{1}{2}<z_1$ close to~$\frac{1}{2}$.
But the link $\beta_{z_0}\sqcup\beta_{z_1}$ is transverse to
$\ker\alpha_1$, and transversely isotopic to $C_0\sqcup C_1$,
whose complement is tight.
\end{proof}
\subsection{Computing the twisting}
\label{subsection:comp-twist}
We now want to show that if we take the minimal $p$ in the above
construction, i.e.\ $p\in\{0,1,2\}$ depending on the values
of $a$ and~$b$, the complement of $\beta_0\sqcup\beta_1$ in $(S^3,\ker\alpha_p)$
will be minimally twisting.
We consider the case $a,b>0$ and $p=2$ illustrated in
Figure~\ref{figure:slopes-hopf}. The other cases are analogous.

First of all, we observe that we have a canonical identification
of $S^3\setminus(\beta_0\sqcup\beta_1)$ with $T^2\times (0,1)$
as in Section~\ref{section:link-complement}, given by identifying
the first $S^1$-factor of $T^2$ with a meridian of~$\beta_0$, and
the second $S^1$-factor with a meridian of~$\beta_1$. This gives
us a well-defined notion of twisting of a contact structure on this
link complement as a positive real number, not just an integer,
cf.\ the discussion in \cite[Section~2.2.1]{hond00I}.

The thickened torus $T^2\times [z_a+\varepsilon,z_b-\varepsilon]$,
where $\varepsilon>0$ is chosen sufficiently small, lies in the
complement of $\beta_0\sqcup\beta_1$. The twisting of $\ker\alpha_p$
on this thickened torus is greater than $\pi/2$, but smaller than~$\pi$.
(In particular, it is minimally twisting in he sense of~\cite{hond00I}.)

Arguing by contradiction, we assume that
$\bigl(S^3\setminus(\beta_0\sqcup\beta_1),\ker\alpha_2\bigr)$ is
not minimally twisting. Under the canonical identification of the
complement of a small open tubular neighbourhood of $\beta_0\sqcup\beta_1$
with $T^2\times[0,1]$, the boundary slope is close to $1/a$ at
$z=0$ and close to $b/1$ at $z=1$. Hence, if $T^2\times[0,1]$
is not minimally twisting, it will have twisting larger than
$3\pi/2$. Indeed, we find two embedded convex tori
(in the isotopy class of $T^2\times\{*\}$) of slope~$0$
and slope~$-\infty$, respectively, bounding a thickened torus of
twisting~$3\pi/2$, see~\cite[Corollary~4.8]{hond00I} and cf.\
Figure~\ref{figure:slopes-hopf}.

Again, this thickened torus of twisting~$3\pi/2$ would persist
in the complement of a tubular neighbourhood of $C_0\sqcup C_1$.
In the tubular neighbourhood of $C_0$ we find a convex torus
with slope $-1/n$ for $n\in\N$ sufficiently large, in the same
isotopy class as the other $2$-tori we are considering in $S^3\setminus
(C_0\sqcup C_1)$. Similarly, near $C_1$ we find such a convex torus with
slope~$-m$ for $m\in\N$ sufficiently large. So the slope would have to
decrease from $-1/n$ to~$0$ via~$-\infty$ and then, after twisting by
$3\pi/2$, from $\infty$ to~$-m$. This would add up to a twisting of more than
$5\pi/2$, which is impossible.
\subsection{Proof of Theorem~\ref{thm:main}~(e)}
\label{subsection:proofe}
We now provide details of the exceptional Legendrian Hopf links where
the contact structure on the link complement has positive twisting.
All claims follow from the preceding discussion.
Together with Proposition~\ref{prop:complement-twist}, which gives
an upper bound on the number of such realisations, this constitutes a
proof of Theorem~\ref{thm:main}~(e). The following three cases cover
all eventualities, possibly after exchanging the roles
of $L_0$ and $L_1$.
\subsubsection{$\ttt_0,\ttt_1>0$}
\label{subsubsection:++}
Here the $\pi$-twisting in the complement equals $p-2$.
For $p\geq 2$ even,
we have the two realisations with $(\tb,\rot)$ equal to
\[ \bigl(\ttt_0,\pm(\ttt_0+1)\bigr)\;\;\;\text{and}\;\;\;
\bigl(\ttt_1,\pm(\ttt_1+1)\bigr)\;\;\;\text{in $\xi_{-1/2}$.}\]
For $p\geq 3$ odd, we can realise
\[ \bigl(\ttt_0,\pm(\ttt_0-1)\bigr)\;\;\;\text{and}\;\;\;
\bigl(\ttt_1,\mp(\ttt_1-1)\bigr)\;\;\;\text{in $\xi_{1/2}$.}\]
For $\ttt_0=\ttt_1=1$, this is only a single realisation.
From Figure~\ref{figure:slopes-hopf} one sees that each individual
component $\beta_0,\beta_1$ is loose, since the contact planes
twist by more than $\pi$ in its complement. For instance, in the complement
of $\beta_1$ in $(S^3,\ker\alpha_2)$ we have the overtwisted disc
\[ \Bigl\{x\in[0,1],\, y=\frac{1}{2},\, 0\leq z\leq \frac{2}{5}\Bigr\}
/\!\sim.\]

The case $p=2$ coincides with the corresponding subcases in the
classification of strongly exceptional Hopf links.
\subsubsection{$\ttt_0,\ttt_1<0$}
The invariants in this and the third case below are the same as
in the first case. The only difference is that
here the $\pi$-twisting in the complement equals~$p$, and one may consider
any $p\geq 0$. For $p=0$, one obtains a realisation in the
tight~$\xist$; for $p\geq 1$ the two link components
are loose. For $p\geq 0$ even and $\ttt_0=\ttt_1=-1$, there
is only one realisation.
\subsubsection{$\ttt_0\leq 0$, $\ttt_1\geq 0$}
Here the twisting is $p-1$, and one may consider any $p\geq 1$.
The components $\beta_0,\beta_1$ are loose, except for the
unknot $\beta_1$ in the case $p=1$ and~$\ttt_1\geq 1$,
which is exceptional by~\cite[Proposition~5.5]{dyma04}.
(Beware that Dymara's notational conventions differ from ours.)

The latter case coincides with the corresponding
examples in Section~\ref{section:exceptional}, as one can see
by looking at the invariants in Table~\ref{table:se-invariants}.
For $\ttt_0=0$ we are in case~(d). For $\ttt_0<0$ and $\ttt_1=0$
we are in case (d) with the roles of $L_0$ and $L_1$ exchanged.
For $\ttt_0<0$ and $\ttt_1=1$ we obtain examples from case~(b2).
For $\ttt_0<0$ and $\ttt_1\geq 2$, one finds the same examples in
case~(b1): when $n$ is even, take $k=-\ttt_0$ and $\ell=0$;
when $n$ is odd, choose $\ell=-\ttt_0$ and $k=0$.
\section{Loose Hopf links}
\label{section:loose}
In this section we prove part (f)
of Theorem~\ref{thm:main}. As shown in~\cite{etny13}, the classification
of loose Legendrian knots in any contact $3$-manifold reduces
to a homotopical question, see also~\cite{geig10}. However,
rather than relying on this general theory, which would have to
be adapted to links, we shall give an \emph{ad hoc}
proof that relies on the topology of the link
complement in our situation.
\subsection{Finding a Legendrian realisation}
We begin by showing that in any given overtwisted $(S^3,\xi_d)$
one can find a loose Legendrian Hopf link realising any
combination of values $(\ttt_0,\ttr_0,\ttt_1,\ttr_1)$
for the classical invariants, subject only to the condition
that $\ttt_i+\ttr_i$ be odd, $i=0,1$.

Start with any topological Hopf link in $(S^3,\xi_d)$ with overtwisted
complement. Construct a loose Legendrian Hopf link as a
$C^0$-approximation of this topological Hopf link;
such an approximation exists by~\cite[Theorem~3.3.1]{geig08}.
Also, the components of this Legendrian Hopf link will
satisfy the parity condition on $\tb+\rot$, since this
is a general phenomenon in contact $3$-manifolds, see
\cite[Remark~4.6.35]{geig08}.

By stabilising the components of this Hopf link, one can increase
or decrease the rotation number in steps of~$1$, while
decreasing the Thurston--Bennequin invariant by~$1$ with
each stabilisation. This is a local process and preserves
the overtwisted disc in the complement. In this way, we can find
a loose Legendrian realisation which has the correct values $\ttr_0,\ttr_1$
of rotation numbers.

We can always decrease the value of $\tb$ of a Legendrian knot
in steps of~$2$, while keeping $\rot$ unchanged, by adding further
pairs of positive and negative stabilisations.

In an overtwisted contact $3$-manifold one can find a loose
Legendrian unknot $K_{1,0}$ with $\tb(K_{1,0})=1$ and
$\rot(K_{1,0})=0$, by forming the connected sum of
two Legendrian knots bounding an overtwisted disc
(with opposite orientations), see~\cite[p.~317]{geig08}.

The connected sum of a Legendrian knot $L$ with $K_{1,0}$ can be performed
in such a way that the topological knot type and $\rot(L)$ remain unchanged,
and such that $\tb(L)$ increases by~$2$; again, see~\cite[p.~317]{geig08}.
In this way, we can adjust
our Legendrian Hopf link so that it also has the correct values
$\ttt_0,\ttt_1$ of Thurston--Bennequin invariants (and remains loose).
\subsection{Uniqueness of the Legendrian realisation}
Let $L_0\sqcup L_1$ and $L_0'\sqcup L_1'$ be two loose
Legendrian realisations of the Hopf link in some overtwisted
$(S^3,\xi_d)$ with $\tb(L_i)=\tb(L_i')$ and $\rot(L_i)=\rot(L_i')$,
$i=0,1$.

The Thurston--Bennequin invariant measures the twisting of the
contact framing relative to the surface framing. These invariants
being equal for $L_i$ and $L_i'$ implies, by the proof
of the Legendrian neighbourhood theorem \cite[Theorem~2.5.8]{geig08},
that we can find a topological diffeotopy $\phi_t$ of $S^3$
starting at $\phi_0=\mathrm{id}_{S^3}$ and ending at a diffeomorphism
$\phi_1$ of $S^3$ that sends an open tubular neighbourhood
$U$ of $L_0\sqcup L_1$ \emph{contactomorphically} to
a neighbourhood of $L_0'\sqcup L_1'$.

We need to show that the contact structures $\xi_d$ and $\phi_{1*}^{-1}\xi_d$
on $S^3\setminus U\cong T^2\times[0,1]$ are diffeomorphic rel boundary
(perhaps for a slightly `thinner'~$U$). By Eliashberg's
classification \cite{elia89} of overtwisted contact structures,
see also \cite[Section~4.7]{geig08}
and the work of Borman--Eliashberg--Murphy~\cite{bem15},
it suffices to show that $\xi_d$ and $\phi_{1*}^{-1}\xi_d$
on $T^2\times[0,1]$ are homotopic rel boundary as tangent $2$-plane fields.

The identification of $S^3\setminus U$ with $T^2\times[0,1]$ is chosen as in
Sections \ref{section:link-complement} and~\ref{subsection:comp-twist}.
That is, the first $S^1$-factor of $T^2$ defines a meridian of~$L_0$,
the second factor a meridian of~$L_1$. As before, we write the
coordinates on $T^2\times[0,1]$ as $x,y\in\R/\Z$ and $z\in[0,1]$.

As an abstract $2$-plane bundle over~$S^3$, the contact structure
$\xi_d$ is trivial, since its Euler class is zero. This allows
us to choose an orthonormal frame $X_1,X_2,X_3$ for the
tangent bundle of~$S^3$, with respect to an auxiliary
Riemannian metric, with $X_1,X_2$ a positive frame for~$\xi_d$,
and $X_3$ positively orthonormal to it.
Likewise, we choose such an orthonormal frame $X_1',X_2',X_3'$
for $\phi_{1*}^{-1}\xi_d$. By writing the vectors in this second frame
in terms of $X_1,X_2,X_3$ we define a map $\Phi\co S^3\rightarrow
\SO(3)$.

We first look at a closed tubular neighbourhood $\nu L_0$ of~$L_0$,
contained in~$U$. This may be chosen as a standard neighbourhood
with convex boundary. In particular, we find a Legendrian longitude
$L_0^{\parallel}$ on the boundary $\partial(\nu L_0)$ that is Legendrian
isotopic to~$L_0$. This can easily be seen in a model
\[ \bigl(S^1\times D^2,\ker(\cos\theta\,\rmd x_1-\sin\theta\,\rmd x_2)\bigr)\]
for~$\nu L_0$, where the longitude $L_0^{\parallel}$ in question is defined by
\[ \theta\longmapsto (\cos\theta,-\sin\theta,\theta).\]
The linking number of $L_0^{\parallel}$ with $L_0$ in $S^3$
equals~$\tb(L_0)$.

On $\nu L_0$, the contact structures $\xi_d$ and $\phi_{1*}^{-1}\xi_d$
coincide, hence so do $X_3$ and~$X_3'$. Along $L_0^{\parallel}$ we can
count the rotations $\ttr,\ttr'$ of the tangent direction
of~$L_0^{\parallel}$ relative to the frames $X_1,X_2$ and $X_1',X_2'$.
These are the rotation numbers of $L_0^{\parallel}$
(or its Legendrian isotopic copy~$L_0$) with respect to
$\xi_d$ and $\phi_{1*}^{-1}\xi_d$, respectively.
From the assumption that $\rot(L_0)=\rot(L_0')$ we conclude
$\ttr=\ttr'$. Thus, after a homotopy of $X_1',X_2'$ we may
assume that this frame coincides with $X_1,X_2$ along~$L_0^{\parallel}$.

Both $X_1,X_2$ and $X_1',X_2'$ define a frame of $\xi_d$
over each meridional disc in $\nu L_0$. Therefore, their
relative twisting along any meridian is zero, and after a further
homotopy of $X_1',X_2'$ we may assume that this frame coincides
with $X_1,X_2$ over all of $\nu L_0$. The same argument can be
applied to~$\nu L_1$.

The map $\Phi$ comparing the two frames $X_1,X_2,X_3$ and
$X_1',X_2',X_3'$ then restricts to a map
\[ \Phi\co\bigl(T^2\times[0,1],T^2\times\{0,1\}\bigr)\longrightarrow
\bigl(\SO(3),\mathrm{id}_{\R^3}\bigr).\]

\begin{lem}
Possibly after a modification of $X_1',X_2'$ near the boundary
$T^2\times\{0,1\}$ that fixes the contact structure
$\phi_{1*}^{-1}\xi_d$, the map $\Phi$ is homotopic
rel boundary to a map that equals the constant map to $\mathrm{id}_{\R^3}$
outside an open $3$-ball in $T^2\times[0,1]$.
\end{lem}

\begin{proof}
The restriction of $\Phi$ to the path $t\mapsto (0,0,t)$, $t\in [0,1]$,
in $T^2\times[0,1]$ defines a loop in $\SO(3)$ based at $\mathrm{id}_{\R^3}$.
Since $\pi_1(\SO(3))=\Z_2$, this may or may not be trivial.
If it is not, we add a full twist to the frame $X_1',X_2'$ relative
to $X_1,X_2$ as we go out in radial direction in $\nu L_0$,
and we undo this twist in $z$-direction in a neighbourhood
of $T^2\times\{0\}$ inside $T^2\times [0,1]$ where $\xi_d$
and $\phi_{1*}^{-1}$ coincide.

After this modification (if necessary) and a homotopy of $X_1',X_2',X_3'$
rel boundary, we may assume that  the two frames coincide along
the described path.

The restriction of $\Phi$ to the cylinder $\bigl\{x=0,\, y\in S^1,\,
z\in [0,1]\bigr\}$, sliced open along the path we just discussed,
defines an element of $\pi_2\bigl(\SO(3)\bigr)=0$. Thus, after yet
another homotopy rel boundary  we may assume that $\Phi=\mathrm{id}_{\R^3}$
also on this cylinder. The same argument applies to the
cylinder $\bigl\{x\in S^1,\, y=0,\, z\in [0,1]\bigr\}$.

We have now achieved that $\Phi$ equals the constant map
to $\mathrm{id}_{\R^3}$ outside the open $3$-ball
\[ \Int(D^3)=(0,1)^3\subset T^2\times[0,1]\subset S^3.\]
In particular, this means that we have homotoped
$\phi_{1*}^{-1}\xi_d$ to a tangent $2$-plane field on $S^3$
that coincides with $\xi_d$ outside $\Int(D^3)$.
\end{proof}

Over $S^3$, the two tangent $2$-plane fields $\xi_d$
and $\phi_{1*}^{-1}\xi_d$ are homotopic, detected by their Hopf invariants
being equal. This means that the obstruction to homotopy
of the two tangent $2$-plane fields over the closed
$3$-ball $D^3$ rel $\partial D^3$ vanishes.
\section{Some computations}
\label{section:compute}
In this section we collect the computations of the classical
invariants for some of the examples in Section~\ref{subsection:Leg-realise}.
The surgery knots in the diagrams are
ordered from top to bottom and
oriented clockwise for the purpose of computing linking and
rotation numbers. The results of all computations are collected
in Table~\ref{table:se-invariants}.

{\small
\begin{sidewaystable}
\vspace*{13cm}
\centering
{\renewcommand{\arraystretch}{1.2}
\begin{tabular}{|c|c|c|c|c|c|c|c|c|}  \hline
Case & Figure               &
$\ttt_0$        & $\ttr_0$      & $\ttt_1$    & $\ttr_1$         &
  type        & $d_3$  & no.\ of             \\
     &                      &
                &               &             &                  &
              &        & realisations\\ \hline\hline
b1   & \ref{figure:b1}      &
$-(k+\ell)<0$   & $\pm(\ell-k-(-1)^n)$ & $n+2\geq 2$ & $\mp(-1)^n(n+1)$ &
  loose/exc.  & $1/2$  & $2|\ttt_0-1|$ \\ \hline
b2   & \ref{figure:b2}      &
$<0$            & $\ttt_0-1,\ttt_0+1,\ldots,-\ttt_0-1,-\ttt_0+1$
                                & 1           & 0                &
  loose/exc.  & $1/2$  & $|\ttt_0-2|$  \\ \hline
c1   & \ref{figure:b2}      &
1               & 0                    & 1           & 0                &
  exc./exc.   & $1/2$  & 1             \\ \hline
c2   & Section~\ref{subsubsection:++} &
2               & $\pm 3$              & 1           & $\pm 2$          &
  loose/loose & $-1/2$ & 2             \\ \cline{2-9}
     & \ref{figure:c2-3-1}  &
3               & $\pm 4$              & 1           & $\pm 2$          &
  loose/loose & $-1/2$ & 2             \\ \cline{3-9}
     &                      &
3               & 0                    & 1           & 0                &
  loose/loose & $3/2$  & 1             \\ \cline{2-9}
     & \ref{figure:c2-2-2}  &
2               & $\pm 3$              & 2           & $\pm 3$          &
  loose/loose & $-1/2$ & 2             \\ \cline{3-9}
     &                      &
2               & $\pm 1$              & 2           & $\pm 1$          &
  loose/loose & $3/2$  & 2             \\ \hline
c3   & \ref{figure:c3-t0-1} &
$\geq 4$        & $\pm(\ttt_0+1)$      & 1           & $\pm 2$          &
  loose/loose & $-1/2$ & 2             \\ \cline{3-9}
     &                      &
$\geq 4$        & $\pm (\ttt_0-3)$     & 1           & 0                &
  loose/loose & $3/2$  & 2             \\ \cline{2-9}
     & \ref{figure:c3-t0-2} &
$\geq 3$        & $\pm(\ttt_0+1)$      & 2           & $\pm 3$          &
  loose/loose & $-1/2$ & 2             \\ \cline{3-9}
     &                      &
$\geq 3$        & $\pm(\ttt_0-1)$      & 2           & $\pm 1$          &
  loose/loose & $3/2$  & 2             \\ \cline{3-9}
     &                      &
$\geq 3$        & $\pm(\ttt_0-3)$      & 2           & $\mp 1$          &
  loose/loose & $3/2$  & 2             \\ \hline
c4   & \ref{figure:c4}      &
$\geq 3$        & $\pm(\ttt_0+1)$      & $\geq 3$    & $\pm(\ttt_1+1)$  &
  loose/loose & $-1/2$ & 2             \\ \cline{3-9}
     &                      &
$\geq 3$        & $\pm(\ttt_0-1)$      & $\geq 3$    & $\pm(\ttt_1-1)$  &
  loose/loose & $3/2$ & 2              \\ \cline{3-9}
     &                      &
$\geq 3$        & $\pm(\ttt_0-3)$      & $\geq 3$    & $\mp(\ttt_1-1)$  &
  loose/loose & $3/2$ & 2              \\ \cline{3-9}
     &                      &
$\geq 3$        & $\pm(\ttt_0-1)$      & $\geq 3$    & $\mp(\ttt_1-3)$  &
  loose/loose & $3/2$ & 2              \\ \hline
d    & \ref{figure:b1}, \ref{figure:b2} &
0               & $\pm 1$              & $\geq 1$   & $\pm(\ttt_1-1)$   &
  loose/exc.  & $1/2$  & 2             \\ \cline{2-9}
     & \ref{figure:d}       &
0               & $\pm 1$              & $\leq 0$   & $\pm(\ttt_1-1)$      &
  loose/loose & $1/2$  & 2             \\ \hline
\end{tabular}
}
\vspace{1.5mm}
\caption{The complete list of strongly exceptional realisations.}
\label{table:se-invariants}
\end{sidewaystable}
} 

\subsection{Figure~\ref{figure:c2-3-1}}
The linking matrix is
\[ \left(\begin{array}{cccc}
-2 & -1 & -1 & -1 \\
-1 & 0  & -1 & -1\\
-1 & -1 & 0  & -1\\
-1 & -1 & -1 & 0
\end{array}\right).\]
The extended linking matrices for $L_0$ and $L_1$ are
\[ M_0=\left(\begin{array}{c|cccc}
0 & 3 & 1 & 1 & 1 \\ \hline
3 &   &   &   &    \\
1 &   & M &   &    \\
1 &   &   &   &    \\
1 &   &   &   & 
\end{array}\right)
\;\;\;\text{and}\;\;\;
M_1=\left(\begin{array}{c|cccc}
0  & -1 & -1 & -1 & -1 \\ \hline
-1 &    &    &    &    \\
-1 &    & M  &    &    \\
-1 &    &    &    &    \\
-1 &    &    &    &
\end{array}\right),\]
respectively. The determinants are
\[ \det M=1,\;\;\; \det M_0=6,\;\;\;\text{and}\;\;\;
\det M_1=2.\]
This gives $\tb(L_0)=3$ and $\tb(L_1)=1$.

For the diagram on the left of Figure~\ref{figure:c2-3-1} we have
\[ \rot(L_0)=\rot_0-\left\langle\begin{pmatrix}2\\0\\0\\0\end{pmatrix},
M^{-1}\begin{pmatrix}3\\1\\1\\1\end{pmatrix}\right\rangle=
-2-\left\langle\begin{pmatrix}2\\0\\0\\0\end{pmatrix},
\begin{pmatrix}-3\\1\\1\\1\end{pmatrix}\right\rangle=4\]
and
\[ \rot(L_1)=\rot_1-\left\langle\begin{pmatrix}2\\0\\0\\0\end{pmatrix},
M^{-1}\begin{pmatrix}-1\\-1\\-1\\-1\end{pmatrix}\right\rangle=
0-\left\langle\begin{pmatrix}2\\0\\0\\0\end{pmatrix},
\begin{pmatrix}-1\\1\\1\\1\end{pmatrix}\right\rangle=2.\]
For the diagram on the right of Figure~\ref{figure:c2-3-1},
the vector $\vrot$ of rotation numbers is the zero vector, so the computation
simplifies and yields $\rot(L_0)=\rot(L_1)=0$.

The signature of $M$ is zero. This can be seen from the Kirby diagram
in Figure~\ref{figure:c2-3-1kirby}. Simply count how many $(\pm 1)$-framed
unknots are left after all blow-downs, minus those introduced to replace
the twisting boxes. There are four $2$-handles, hence $\chi=5$.
For the diagram on the left, the solution $\bfx$ of $M\bfx=\vrot$ is
$\bfx=(-4,2,2,2)^{\ttt}$; on the right, $\bfx$ is the zero vector.
This gives $c^2=-8$ in the first and $c^2=0$ in the second case,
resulting in the values of $d_3$ listed in Table~\ref{table:se-invariants}.
The contact structure described by the surgery diagram on the left is the
overtwisted structure $\xi_{-1/2}$ (rather than $\xist$, which also
has $d_3=-\frac{1}{2}$), since $L_0$ and $L_1$ violate the
Bennequin inequality \cite[Theorem~4.6.36]{geig08} for Legendrian knots
in tight contact $3$-manifolds.

The Legendrian Hopf link $L_0\sqcup L_1$ is exceptional, since
contact $(-1)$-surgery along $L_0$ and contact $(-1)$-surgery along
$L_1$ and two Legendrian push-offs of $L_1$ leads, by the cancellation
lemma, to the empty diagram, that is, $(S^3,\xist)$. The individual
components $L_0,L_1$ are loose, however, since there are no
exceptional unknots in $(S^3,\xi_{-1/2})$ or $(S^3,\xi_{3/2})$.
The same arguments apply to the examples in Figure~\ref{figure:c2-2-2}.
\subsection{Figure~\ref{figure:c3-t0-1}}
Here we only collect the data the reader needs to verify
our claim about the invariants in this case listed in
Table~\ref{table:se-invariants}. As linking matrix $M$ we take the one
given by ordering the surgery curves from top to bottom, all oriented
clockwise. The signature of this matrix is $\sigma=-n-1$.
This can be read off from the Kirby moves as explained in the previous
example. The Euler characteristic is $\chi=n+6$.

Rather than computing the Thurston--Bennequin invariants of $L_0$ and $L_1$
from the extended linking matrices, one can simply keep control over
the contact framing during the Kirby moves described in
Figures \ref{figure:c3-t0-1kirby} and~\ref{figure:b1kirby}.
We find $\tb(L_0)=n+4$ and $\tb(L_1)=1$.

For the figure on the left, the vector $\bfx$ with $M\bfx=\vrot$ is
\[ \bfx=\bigl(-(n+1),n,-(n-1),\ldots,-1,0,0,0,0\bigr)^{\ttt}\;\;\;
\text{for $n\geq 0$ even}\]
and
\[ \bfx=\bigl(-(n+5),n+4,-(n+3),\ldots,5,-4,2,2,2\bigr)^{\ttt}\;\;\;
\text{for $n\geq 1$ odd.}\]
This gives $d_3=\frac{3}{2}$ for $n$ even, and $d_3=-\frac{1}{2}$
for $n$ odd.

For the figure on the right, we have
\[ \bfx=\bigl(n+5,-(n+4),\ldots,5,-4,2,2,2\bigr)^{\ttt}\;\;\;
\text{for $n$ even}\]
and
\[ \bfx=\bigl(n+1,-n,n-1,\ldots,-1,0,0,0,0\bigr)^{\ttt}\;\;\;
\text{for $n$ odd,}\]
leading to $d_3=-\frac{1}{2}$ for $n$ even and $d_3=\frac{3}{2}$
for $n$ odd.

For the clockwise orientation of $L_0$ and $L_1$ (beware that for $n$
odd this gives a negative Hopf link) we have
\[ \vlk_0=(-2,-1,0,\ldots,0)^{\ttt}\]
and
\[ \vlk_1=(1,\underbrace{0,\ldots,0}_n,-1,-1,-1,-1)^{\ttt}\]
for the vector of linking numbers of $L_0$ and $L_1$, respectively,
with the surgery curves. This yields
\[ M^{-1}\vlk_0=\bigl(n+4,-(n+2),n+1,\ldots,(-1)^n3,(-1)^{n+1}2,
(-1)^n,(-1)^n,(-1)^n\bigr)^{\ttt}\]
and
\[ M^{-1}\vlk_1=\bigl((-1)^n,(-1)^{n-1},(-1)^{n-2},\ldots,-1,1,1,1
\bigr)^{\ttt}.\]
The remaining calculations are then straightforward.

The arguments for the looseness of the components and the exceptionality
of the link are as in the previous case.
\subsection{Figure~\ref{figure:d}}
The linking matrix is the $\bigl((n+1)\times(n+1)\bigr)$-matrix
\[ M=M^{(n)}=\begin{pmatrix}
-1     & -1      & -1     & \cdots & -1\\
-1     & 0       & -1     & \cdots & -1\\
-1     & -1      & \ddots & \cdots & -1\\
\vdots &  \vdots &        & \cdots & \vdots \\
-1     & -1      & -1     & \cdots & 0
\end{pmatrix} \]
The extended linking matrices are
\[ M^{(n)}_0=\left(\begin{array}{c|ccc}
0      & -1 & \cdots   & -1\\ \hline
-1     &    &          &   \\
\vdots &    & M^{(n)} &   \\
-1     &    &          &
\end{array}\right) \]
and
\[ M^{(n)}_1=\left(\begin{array}{c|cccc}
0      & -2 & -1 & \cdots   & -1\\ \hline
-2     &    &    &          &   \\
-1     &    &    &          &   \\
\vdots &    &    & M^{(n)}  &   \\
-1     &    &    &          &
\end{array}\right). \]
One easily computes $\det M^{(n)}=-1$, $\det M^{(n)}_0=-1$,
and $\det M^{(n)}_1=n-4$. With formula~(\ref{eqn:tb})
this yields $\tb(L_0)=0$ and $\tb(L_1)=2-n$.

The Euler characteristic is $\chi=n+2$, and
from the Kirby moves in Figure~\ref{figure:dkirby} we read off
$\sigma=n-1$. The solution of $M^{(n)}\bfx=\vrot=(1,0,\ldots,0)^{\ttt}$ is
\[ \bfx=(n-1,-1,\ldots,-1)^{\ttt},\]
hence $c^2=\langle\bfx,\vrot\rangle=n-1$.
With (\ref{eqn:d3surgery}) this gives $d_3=\frac{1}{2}$.

For the rotation number we observe that with
\[ \vlk_0=(-1,\ldots,-1)^{\ttt}\]
and
\[ \vlk_1=(-2,-1,\ldots,-1)^{\ttt}\]
we have
\[ M^{-1}\vlk_0=(1,0,\ldots,0)^{\ttt}\]
and
\[ M^{-1}\vlk_1=(2-n,1,\ldots,1)^{\ttt}.\]
With (\ref{eqn:rot}) we obtain $\rot(L_0)=-1$ and
$\rot(L_1)=n-1$.


\begin{thebibliography}{10}
%
\bibitem{bem15}
\textsc{M. S. Borman, Ya. Eliashberg and E. Murphy},
Existence and classification of overtwisted contact structures
in all dimensions,
\textit{Acta Math.}
\textbf{215} (2015), 281--361.
%
\bibitem{dige04}
\textsc{F. Ding and H. Geiges},
A Legendrian surgery presentation of contact $3$-manifolds,
\textit{Math. Proc. Cambridge Philos. Soc.}
\textbf{136} (2004), 583--598.
%
\bibitem{dige07}
\textsc{F. Ding and H. Geiges},
Legendrian knots and links classified by classical invariants,
\textit{Commun. Contemp. Math.}
\textbf{9} (2007), 135--162.
%
\bibitem{dige10}
\textsc{F. Ding and H. Geiges},
Legendrian helix and cable links,
\textit{Commun. Contemp. Math.}
\textbf{12} (2010), 487--500.
%
\bibitem{dgs04}
\textsc{F. Ding, H. Geiges and A. I. Stipsicz},
Surgery diagrams for contact $3$-manifolds,
\textit{Turkish J. Math.}
\textbf{28} (2004), 41--74.
%
\bibitem{dgs05}
\textsc{F. Ding, H. Geiges and A. I. Stipsicz},
Lutz twist and contact surgery,
\textit{Asian J. Math.}
\textbf{9} (2005), 57--64.
%
\bibitem{dlz13}
\textsc{F. Ding, Y. Li and Q. Zhang},
Tight contact structures on some bounded Seifert manifolds
with minimal convex boundary,
\textit{Acta Math. Hungar.}
\textbf{139} (2013), 64--84.
%
\bibitem{dyma04}
\textsc{K. Dymara},
Legendrian knots in overtwisted contact structures,
\texttt{arXiv:math/0410122v2}.
%
\bibitem{elia89}
\textsc{Ya. Eliashberg},
Classification of overtwisted contact structures on $3$-manifolds,
\textit{Invent. Math.}
\textbf{98} (1989), 623--637.
%
\bibitem{elfr09}
\textsc{Ya. Eliashberg and M. Fraser},
Topologically trivial Legendrian knots,
\textit{J. Symplectic Geom.}
\textbf{7} (2009), 77--127.
%
\bibitem{etny05}
\textsc{J. B. Etnyre},
Legendrian and transversal knots,
\textit{Handbook of Knot Theory},
Elsevier, Amsterdam (2005), 105--185.

%
\bibitem{etny13}
\textsc{J. B. Etnyre},
On knots in overtwisted contact structures,
\textit{Quantum Topol.}
\textbf{4} (2013), 229--264.
%
\bibitem{etny-class}
\textsc{J. B. Etnyre},
Convex surfaces in contact geometry, class notes,
available at\\
\verb+http://people.math.gatech.edu/~etnyre/preprints/notes.html+
%
\bibitem{etho01}
\textsc{J. B. Etnyre and K. Honda},
Knots and contact geometry~I. Torus knots and the figure eight knot,
\textit{J. Symplectic Geom.}
\textbf{1} (2001), 63--120.
%
\bibitem{geig08}
\textsc{H. Geiges},
\textit{An Introduction to Contact Topology},
Cambridge Stud. Adv. Math. \textbf{109}
(Cambridge University Press, Cambridge, 2008).
%
\bibitem{geig10}
\textsc{H. Geiges},
A note on Legendrian knot complements,
unpublished note, dated 11 January 2010;
available upon request.
%
\bibitem{geon15}
\textsc{H. Geiges and S. Onaran},
Legendrian rational unknots in lens spaces,
\textit{J. Symplectic Geom.}
\textbf{13} (2015), 17--50.
%
\bibitem{geon}
\textsc{H. Geiges and S. Onaran},
Exceptional Legendrian torus knots,
\textit{Int. Math. Res. Not. IMRN}, to appear.
%
\bibitem{giro94}
\textsc{E. Giroux},
Une structure de contact, m\^eme tendue, et plus ou moins tordue,
\textit{Ann. Sci. \'Ecole Norm.\ Sup. (4)}
%
\textbf{27} (1994), 697--705.
%
%
\bibitem{giro99}
\textsc{E. Giroux},
Une infinit\'e de structures de contact tendues sur une
infinit\'e de vari\'et\'es,
\textit{Invent. Math.}
\textbf{135} (1999), 789--802.
%
\bibitem{giro00}
\textsc{E. Giroux},
Structures de contact en dimension trois et bifurcations des
feuilletages de surfaces,
\textit{Invent. Math.}
\textbf{141} (2000), 615--689.
%
\bibitem{gomp98}
\textsc{R. E. Gompf},
Handlebody construction of Stein surfaces,
\textit{Ann. of Math. (2)}
%
%
\textbf{148} (1998), 619--693.
%
\bibitem{gost99}
\textsc{R. E. Gompf and A. I. Stipsicz},
\textit{$4$-Manifolds and Kirby Calculus},
Grad. Stud. Math. \textbf{20},
American Mathematical Society, Providence, RI (1999).
%
\bibitem{hond00I}
\textsc{K. Honda},
On the classification of tight contact structures~I,
\textit{Geom. Topol.}
\textbf{4} (2000), 309--368;
erratum: \textit{Geom.\ Topol.}
\textbf{5} (2001), 925--938.
%
\bibitem{hond00II}
\textsc{K. Honda},
On the classification of tight contact structures~II,
\textit{J. Differential Geom.}
\textbf{55} (2000), 83--143.
%
\bibitem{lerm01}
\textsc{E. Lerman},
Contact cuts,
\textit{Israel J. Math.}
\textbf{124} (2001), 77--92.
%
\bibitem{loss09}
\textsc{P. Lisca, P. Ozsv\'ath, A. I. Stipsicz and Z. Szab\'o},
Heegaard Floer invariants of Legendrian knots in contact three-manifolds,
\emph{J. Eur. Math. Soc. (JEMS)}
{\bf 11} (2009), 1307--1363.
%
\bibitem{voge18}
\textsc{T. Vogel},
Non-loose unknots, overtwisted discs, and the contact mapping class
group of~$S^3$,
\textit{Geom. Funct. Anal.}
\textbf{28} (2018), 228--288.
%
\end{thebibliography}
\end{document}